\documentclass[12pt, reqno]{amsart}
\usepackage{amsmath,times,epsfig,amssymb,amsbsy,amscd,amsfonts,amstext,color,bm}

\usepackage{bm}

\usepackage{comment}
\usepackage{tikz}

\numberwithin{equation}{section}

\theoremstyle{plain}
\newtheorem{theorem}{Theorem}[section]
\newtheorem{corollary}[theorem]{Corollary}
\newtheorem{lemma}[theorem]{Lemma}

\newtheorem{proposition}[theorem]{Proposition}

\theoremstyle{definition}
\newtheorem{definition}[theorem]{Definition}
\newtheorem{example}[theorem]{Example}

\theoremstyle{remark}
\newtheorem{remark}[theorem]{Remark}

\newcommand{\A}{\mathcal{A}}
\newcommand{\B}{\mathcal{B}}

\newcommand{\F}{\mathbb{F}}
\newcommand{\Z}{\mathbb{Z}}

\newcommand{\R}{\mathbb{R}}
\newcommand{\C}{\mathbb{C}}
\newcommand{\K}{\mathbb{K}}

\newcommand{\scC}{\mathcal{C}}
\newcommand{\scF}{\mathcal{F}}
\newcommand{\scS}{\mathcal{S}}
\newcommand{\scU}{\mathcal{U}}
\newcommand{\scL}{\mathcal{L}}

\newcommand{\Br}{\operatorname{Br}}

\newcommand{\OS}{\operatorname{OS}}
\newcommand{\Poin}{\operatorname{Poin}}

\newcommand{\Aut}{\operatorname{Aut}}
\newcommand{\bch}{\operatorname{bch}}
\newcommand{\ch}{\operatorname{ch}}

\newcommand{\codim}{\operatorname{codim}}

\newcommand{\Crit}{\operatorname{Crit}}

\newcommand{\DF}{\operatorname{DF}}

\newcommand{\eps}{\varepsilon}
\newcommand{\Gal}{\operatorname{Gal}}
\newcommand{\grad}{\operatorname{grad}}
\newcommand{\Hilb}{\operatorname{Hilb}}
\newcommand{\Hom}{\operatorname{Hom}}

\newcommand{\Ker}{\operatorname{Ker}}
\newcommand{\Mod}{\operatorname{Mod}}
\newcommand{\obj}{\operatorname{obj}}
\newcommand{\rank}{\operatorname{rank}}
\newcommand{\Sal}{\operatorname{Sal}}
\newcommand{\Sep}{\operatorname{Sep}}

\newcommand{\sign}{\operatorname{sign}}

\begin{document}

\title[Hyperplane arrangements]{Topology of hyperplane arrangements via real structure}

\begin{abstract}
This note is a survey on the topology of hyperplane arrangements. 
We mainly focus on the relationship between topology and the 
real structure, such as adjacent relations of chambers and stratifications 
related to real structures. 
\end{abstract}

\author{Masahiko Yoshinaga}
\address{Masahiko Yoshinaga, 
Osaka University}
\email{yoshinaga@math.sci.osaka-u.ac.jp}





\date{\today}

\maketitle

\tableofcontents

\section{Introduction}

Let $\A=\{H_1, H_2, \dots, H_n\}$ be a set of hyperplanes 
in $\R^\ell$. On one hand, $\A$ provides a stratification of 
the space $\R^\ell$. On the other hand, the complexification 
gives a complex hyperplane arrangement in $\C^\ell$. 
The complement $M(\A)=\C^\ell\setminus\bigcup_{H\in\A}
H\otimes\C$ is a connected affine variety, which is a 
real $2\ell$-dimensional manifold. The topology of $M(\A)$ 
is known to be closely related to the structure of $\A$. 
In this note, we discuss the relationship between the topology of 
$M(\A)$ and its real structure, focusing on the adjacent relations 
of chambers.

\section{Hyperplane arrangements over general fields}

\subsection{Braid arrangements and graphical arrangements}
\label{sec:braid}

As a typical example, we start this note 
with the \emph{braid arrangements}. 
Let $V=\K^\ell$. 
Let $H_{ij}=\{(x_1, \dots, x_\ell)\in\K^\ell\mid x_i-x_j=0\}$ for 
$1\leq i<j\leq \ell$. The collection of these hyperplanes is 
called the \emph{braid arrangement} denoted by 
\[
\Br(\ell)=\{H_{ij}\mid 1\leq i<j\leq\ell\}. 
\]

\begin{example}
\label{ex:rbraid}
(Braid arrangements over $\R$.) 
If $\K=\R$, 
the complement of the braid arrangement $\Br(\ell)$ 
consists of finitely many connected components called 
\emph{regions}, or \emph{chambers}. Each chamber is 
expressed as 
\[
\{(x_1, \dots, x_\ell)\mid x_{\sigma(1)}<\dots<x_{\sigma(\ell)}\}, 
\]
where $\sigma\in\frak{S}_\ell$ a permutation. Therefore, there 
are exactly $\ell!$ chambers. 
\end{example}

\begin{example}
\label{ex:fbraid}
(Braid arrangements over $\F_q$.) 
If $\K=\F_q$, 
the complement of the braid arrangement $\Br(\ell)$ 
\[
\F_q^\ell\setminus\bigcup_{1\leq i<j\leq\ell}H_{ij}
\]
consists of ordered tuples $(x_1, \dots, x_\ell)$ 
$\ell$ mutually distinct elements in $\F_q$. 
Therefore, there are exactly $q(q-1) \dots (q-\ell+1)$ elements. 
\end{example}

\begin{example}
\label{ex:cbraid}
(Braid arrangements over $\C$. \cite{fad-neu, fox-neu, arn})
If $\K=\C$, the complement of the braid arrangement 
\[
X_\ell=\{(x_1, \dots, x_\ell)\mid x_i\in\C, x_i\neq x_j (i\neq j)\}
\]
has a natural fibration 
\[
p: X_\ell\longrightarrow X_{\ell-1}, (x_1, \dots, x_\ell)
\longmapsto (x_1, \dots, x_{\ell-1})
\]
with the fiber $\C\setminus\{\mbox{$\ell$ points}\}$. 
Using this fibration, one can prove $X_\ell$ is a $K(\pi, 1)$-space. 
Also, the fibration is homologically trivial, hence one can apply 
Leray-Hirsch theorem to obtain the Poincar\'e polynomial 
\[
\Poin(X_\ell, t)=(1+t)(1+2t)\dots (1+(\ell-1)t). 
\]
\end{example}

Even though the nature of the problems in 
Example \ref{ex:rbraid}, 
Example \ref{ex:fbraid}, and 
Example \ref{ex:cbraid} is different, the results 
(the numbers of chambers, points, and the Poincar\'e polynomial) 
look similar. This phenomenon is explained (in \S \ref{sec:basics}) 
by the fact that 
these invariants are determined solely by the intersection poset and 
the characteristic polynomial which are the same for any 
field $\K$. 

\subsection{Definitions and classical results}
\label{sec:basics}

Let $V=\K^\ell$ be a linear space. A subset of $H$ of $V$ 
of the form 
\[
H=\{x\in V\mid \alpha(x)=b\}, 
\]
where $\alpha:V\to\K$ is a non-zero linear form and 
$b\in\K$, is called an affine hyperplane. A set 
$\A=\{H_\lambda\}_{\lambda\in\Lambda}$ of hyperplanes is 
called an \emph{(affine) hyperplane arrangement}, or simply 
an \emph{arrangement}. 
An arrangement $\A$ is called \emph{central} if 
$\bigcap_{H\in\A}H\neq\emptyset$ and contains the origin 
$\bm{0}$. 
An arrangement $\A$ is called \emph{essential} if 
there exist $H_1, \dots, H_\ell\in\A$ such that 
$\dim (H_1\cap\cdots\cap H_\ell)=0$. 

\subsection{Intersection posets, the M\"obius function, and chambers}

\begin{definition}
Let $\A$ be an arrangement. Denote the set 
of non-empty intersections of hyperplanes by 
\[
L(\A)=\left\{\left. \bigcap_{H\in\B}H\neq\emptyset
\right|\ \B\subset\A \right\}. 
\]
Note that if $\B=\emptyset$, $\bigcap_{H\in\B}H=V$. 
We consider $L(\A)$ as a poset equipped with the 
reverse inclusion order (Figure \ref{fig:poset}). We also 
consider the set of all $k$-dimensional intersections, 
\[
L_k(\A)=\{X\in L(\A)\mid \dim X=k\}. 
\]
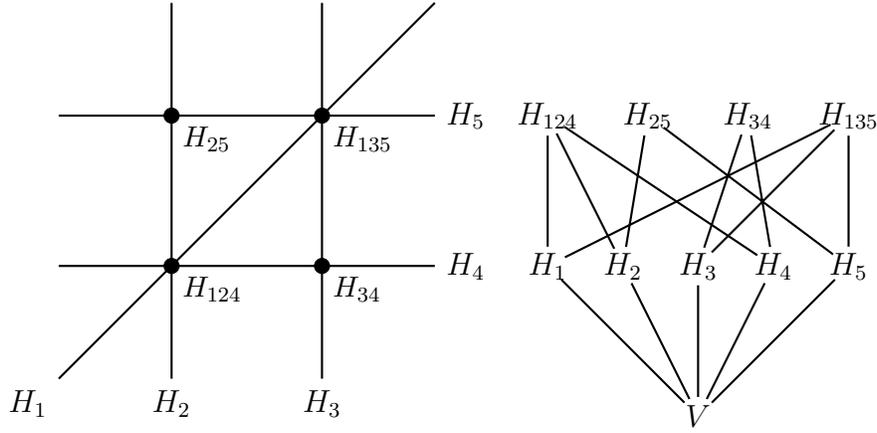
\begin{figure}[htbp]
\centering
\begin{tikzpicture}


\coordinate (C1) at (1.5,0.75);
\coordinate (C2) at (0.75,1.5);
\coordinate (C3) at (1,3);
\coordinate (C4) at (1,5);

\coordinate (C5) at (3,1);
\coordinate (C6) at (3.5,2.5);
\coordinate (C7) at (2.5,3.5);
\coordinate (C8) at (3,5);

\coordinate (C9) at (5,1);
\coordinate (C10) at (5,3);
\coordinate (C11) at (5.25,4.5);
\coordinate (C12) at (4.5,5.25);

\draw[thick] (0.5,0.5) node[anchor=north east] {$H_1$} -- ++(5,5); 
\draw[thick] (2,0.5) node[anchor=north] {$H_2$} -- ++(0,5); 
\draw[thick] (4,0.5) node[anchor=north] {$H_3$} -- ++(0,5); 
\draw[thick] (0.5,2) -- ++(5,0) node[anchor=west] {$H_4$}; 
\draw[thick] (0.5,4) -- ++(5,0) node[anchor=west] {$H_5$}; 

\filldraw[fill=black, draw=black] (2,2) node[anchor=north west] {\small $H_{124}$} circle (0.1); 
\filldraw[fill=black, draw=black] (2,4) node[anchor=north west] {\small $H_{25}$} circle (0.1); 
\filldraw[fill=black, draw=black] (4,2) node[anchor=north west] {\small $H_{34}$} circle (0.1); 
\filldraw[fill=black, draw=black] (4,4) node[anchor=north west] {\small $H_{135}$} circle (0.1); 

\coordinate (V) at (9,0);
\coordinate (H1) at (7,2);
\coordinate (H2) at (8,2);
\coordinate (H3) at (9,2);
\coordinate (H4) at (10,2);
\coordinate (H5) at (11,2);
\coordinate (H124) at (7,4);
\coordinate (H25) at (8.33,4);
\coordinate (H34) at (9.66,4);
\coordinate (H135) at (11,4);

\draw[thick] (V)--(H1);
\draw[thick] (V)--(H2);
\draw[thick] (V)--(H3);
\draw[thick] (V)--(H4);
\draw[thick] (V)--(H5);
\draw[thick] (H1)--(H124)--(H2)--(H25)--(H5)--(H135)--(H3)--(H34)--(H4);
\draw[thick] (H4)--(H124);
\draw[thick] (H1)--(H135);

\filldraw[fill=white, draw=white]  (V) node {$V$} circle (0.25); 
\filldraw[fill=white, draw=white]  (H1) node {$H_1$} circle (0.25); 
\filldraw[fill=white, draw=white]  (H2) node {$H_2$} circle (0.25); 
\filldraw[fill=white, draw=white]  (H3) node {$H_3$} circle (0.25); 
\filldraw[fill=white, draw=white]  (H4) node {$H_4$} circle (0.25); 
\filldraw[fill=white, draw=white]  (H5) node {$H_5$} circle (0.25); 
\filldraw[fill=white, draw=white]  (H124) node {$H_{124}$} circle (0.26); 
\filldraw[fill=white, draw=white]  (H25) node {$H_{25}$} circle (0.26); 
\filldraw[fill=white, draw=white]  (H34) node {$H_{34}$} circle (0.26); 
\filldraw[fill=white, draw=white]  (H135) node {$H_{135}$} circle (0.26); 

\end{tikzpicture}
\caption{A line arrangement and the intersection poset} 
\label{fig:poset}
\end{figure}
\end{definition}

\begin{remark}
The intersection poset 
$L(\A)$ is graded in the sense that lengths of 
maximal chains are equal. 

We also note that $\A$ is essential if and only if 
$L_0(\A)\neq\emptyset$. 
\end{remark}

Let $X\in L(\A)$. The arrangement $\A$ naturally 
defines an arrangement on $X$ 
\[
\{X\cap H\mid H\in\A, H\not\supset X\}
\]
which we denote by $\A^X$ and call the restriction. 
We also define 
\[
\A_X:=\{H\in\A\mid H\supset X\}
\]
which is called the localization at $X$. It is clear that 
$L(\A^X)\simeq L(\A)_{\geq X}$ and 
$L(\A_X)\simeq L(\A)_{\leq X}$. 

If the arrangement $\A$ is defined in $\R^\ell$, then the 
complement 
$\R^\ell\setminus\bigcup_{i=1}^nH_i$ is a disjoint union of 
open convex subsets. Each connected component is 
called a \emph{chamber}. 
We denote the set of chambers by $\ch(\A)$. We also denote 
the set of \emph{bounded chambers} by $\bch(\A)$. 

\subsection{Characteristic polynomials}

Define the M\"obius function $\mu:L(\A)\to\Z$ by 
\begin{equation}
\mu(X)=
\begin{cases}
1 & \mbox{ if }X=V\\
-\sum_{Z<X}\mu(Z) & \mbox{ if }X\neq V. 
\end{cases}
\end{equation}

\begin{definition}
The characteristic polynomial $\chi(\A, t)$ is defined as 
\[
\chi(\A, t)=\sum_{X\in L(\A)}\mu(X)t^{\dim X}. 
\]
\end{definition}
The characteristic polynomial has other expressions. 
\begin{proposition}
\label{prop:add-del}
Let $\A=\{H_1, \dots, H_n\}$ be an arrangement. 

(1) 
\[
\chi(\A, t)=\sum_{\substack{I\subset [n]\\ H_I\neq\emptyset}}
(-1)^{\# I}t^{\dim H_I}, 
\]
where $H_I=\bigcap_{i\in I}H_i$. 

(2) 
\begin{equation}
\chi(\A, t)=
\begin{cases}
t^{\dim V}, &\mbox{ if }\A=\emptyset, \\
\chi(\A\setminus\{H\}, t)-\chi(\A^H, t), 
&\mbox{ if }H\in\A\neq\emptyset. 
\end{cases}
\end{equation}
\end{proposition}
If the arrangement $\A$ is defined over $\R$, then it decomposes 
the space $\R^\ell$ into open convex connected components. 
Such a connected component is called a \emph{chamber} and the 
set of all chambers is denoted by $\ch(\A)$. 
If $\A$ is essential, there may exist \emph{bounded chambers}. 
The set of all bounded chambers is denoted by $\bch(\A)$. 

As applications of Proposition \ref{prop:add-del}, we have 
the following. 
\begin{proposition}
\label{prop:chicount}
(1) (Zaslavsky) 
Let $\A=\{H_1, \dots,H_n\}$ be an arrangement in $V=\R^\ell$. 
Then 
\[
\#\ch(\A)=(-1)^\ell\chi(\A, -1). 
\]
Furthermore, if $\A$ is essential, then 
\[
\#\bch(\A)=(-1)^\ell\chi(\A, 1). 
\]
(2) (Crapo-Rota) 
Let $\A=\{H_1, \dots,H_n\}$ be an arrangement in $V=\F_q^\ell$. 
Then 
\[
\# M(\A)=\chi(\A, q). 
\]
\end{proposition}

\begin{example}
\label{ex:A3}
Let $\A=\{H_1, H_2, H_3, H_4, H_5\}$ be $5$ lines in $\R^2$ as shown in 
Figure \ref{fig:poset}. Then $\mu(H_i)=-1$ for each $i\in[5]$. 
By the definition of M\"obius function, we also have 
$\mu(H_{25})=\mu(H_{34})=1$ and 
$\mu(H_{124})=\mu(H_{135})=2$. Hence, we have 
\[
\chi(\A, t)=t^2-5t+6=(t-2)(t-3). 
\]
Also, we have $\chi(\A, -1)=12=\#\ch(\A)$, and 
$\chi(\A, 1)=2=\#\bch(\A)$. 
\end{example}

\subsection{Chromatic polynomial}

An interesting class of arrangements is the so-called 
\emph{graphical arrangement}. In this class, the characteristic 
polynomial is directly connected to enumerative problems 
\cite{sta-ec1, sta-hyp}. 

Let $([\ell], E)$ be a finite simple graph, where 
$[\ell]=\{1, \dots, \ell\}$ is the vertex set, and 
$E\subset\{\{i, j\}\mid i, j\in[\ell], i\neq j\}$ is the 
set of edges. 

\begin{definition}
Let $G=([\ell], e)$ be a graph and 
$n$ be a positive integer. 
A map $c: [\ell]\longrightarrow [n]$ is called a 
\emph{vertex coloring} (with $n$-colors) 
of $G$ if $c(i)\neq c(j)$ whenever 
$\{i, j\}\in E$. The number of vertex colorings 
with $n$-colors is denoted by 
\[
\chi_G(n):=\#\{c: [\ell]\longrightarrow [n]\mid c\mbox{ is a coloring}\}. 
\]
\end{definition}
As we will see below, the function 
$\chi_G(n)$ is known to be a polynomial in $n$, 
called the \emph{chromatic polynomial}. 
Furthermore, it is the characteristic polynomial of the following 
graphical arrangement. 
\begin{definition}
Let $G=([\ell], e)$ be a graph. Define the \emph{graphical arrangement} 
$\A_G$ by 
\[
\A_G=\{H_{ij}\mid \{i, j\}\in E\}. 
\]
\end{definition}
Note that if $G$ is the complete graph, then $\A_G$ is equal to 
$\Br(\ell)$. Hence $\A_G$ is a subarrangement of the graphical 
arrangement. 
\begin{proposition}
Let $G=([\ell], e)$ be a graph. Then $\chi_G(t)=\chi(\A_G, t)$. 
\end{proposition}
\begin{proof}
The proof is by induction on the number of edges. 
In the case $E=\emptyset$, $\A_G$ is the empty arrangement. 
Then it is clear that 
$\chi_G(t)=\chi(\A_G, t)=t^\ell$. Suppose that $E\neq\emptyset$. 
Choose an edge $e=\{i, j\}\in E$ and consider the corresponding 
hyperplane $H_e=H_{ij}$. Compare the set of colorings of 
graphs $G$, the deletion $G\smallsetminus e$, and the contraction 
$G/e$ (Figure \ref{fig:dc}). 
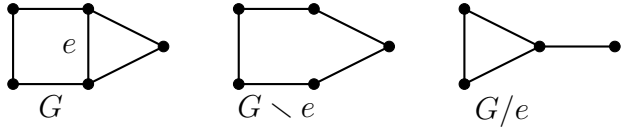
\begin{figure}[htbp]
\begin{tikzpicture}

\filldraw[fill=black, draw=black] (0,0) circle (2pt) ;
\filldraw[fill=black, draw=black] (1,0) circle (2pt) ;
\filldraw[fill=black, draw=black] (0,1) circle (2pt) ;
\filldraw[fill=black, draw=black] (1,1) circle (2pt) ;
\filldraw[fill=black, draw=black] (2,0.5) circle (2pt) ;
\draw[thick] (1,0) -- node[left]{$e$} (1,1) --(0,1) --
(0,0) -- node[below] {$G$} (1,0) --(2,0.5)--(1,1);

\filldraw[fill=black, draw=black] (3,0) circle (2pt) ;
\filldraw[fill=black, draw=black] (4,0) circle (2pt) ;
\filldraw[fill=black, draw=black] (3,1) circle (2pt) ;
\filldraw[fill=black, draw=black] (4,1) circle (2pt) ;
\filldraw[fill=black, draw=black] (5,0.5) circle (2pt) ;
\draw[thick] (4,0) -- (5,0.5) -- (4,1) --(3,1)--
(3,0) -- node[below] {$G\smallsetminus e$} (4,0);

\filldraw[fill=black, draw=black] (6,0) circle (2pt) ;
\filldraw[fill=black, draw=black] (6,1) circle (2pt) ;
\filldraw[fill=black, draw=black] (7,0.5) circle (2pt) ;
\filldraw[fill=black, draw=black] (8,0.5) circle (2pt) ;
\draw[thick] (7,0.5)--(6,1)--(6,0) -- (7,0.5) -- (8,0.5);
\draw (6.5,0) node[below] {$G/ e$} ;

\end{tikzpicture}
\caption{The deletion $G\smallsetminus e$ and the contraction $G/e$}
\label{fig:dc}
\end{figure}
Let $c:[\ell]\longrightarrow[n]$ be a coloring of the deletion 
$G\smallsetminus e$. The set of colorings of $G\smallsetminus e$ 
is decomposed into two subsets according to 
$c(i)\neq c(j)$ or $c(i)=c(j)$. 
If $c(i)\neq c(j)$, it is equivalent to a coloring of $G$, and 
if $c(i)=c(j)$, it is equivalent to a coloring of $G/e$. 
Thus we have 
\[
\{\mbox{Coloring of }G\smallsetminus e\}=
\{\mbox{Coloring of }G\}\sqcup
\{\mbox{Coloring of }G/ e\}, 
\]
This leads to the following recursive formula 
\begin{equation}
\chi_G(n)=\chi_{G\smallsetminus e}(n)-\chi_{G/e}(n). 
\end{equation}
We also note that the graphical arrangement of the 
contraction $G/e$ can be identified with the restriction of 
$A_G$ on $H_e$, in other words, $A_{G/e}=(\A_G)^{H_e}$. 
Thus we have 
\[
\begin{split}
\chi_G(t)
&=
\chi_{G\smallsetminus e}(t)-\chi_{G/e}(t)\\
&=
\chi(A_{G\smallsetminus e}, t)-\chi(A_{G/e}, t)\\
&=
\chi(\A_G\setminus\{H_e\}, t)-\chi((\A_G)^{H_e}, t)\\
&=
\chi(\A, t). 
\end{split}
\]
\end{proof}
As we will see, the coefficients of 
the characteristic polynomial are equal to the Betti numbers 
of complement of a complex hyperplane arrangement. 
Recently, a generalization, \emph{characteristic quasi-polynomials} 
has also been intensively studied in connection with 
topology of toric arrangements in $(\C^\times)^\ell$ 
\cite{ktt, moc-ari, lty}.

\subsection{Orlik-Solomon algebras} 

Let $\A=\{H_1, \dots, H_n\}$ be an arrangement. 
Let $E=\bigcup_{i=1}^n\Z e_i$ be a free $\Z$-module generated 
by $e_1, \dots, e_n$ and let $\wedge E$ be its exterior algebra. 
For an (ordered) subset $I=\{i_1, i_2, \dots, i_p\}
\subset [n]$, denote $e_I=e_{i_1}\wedge\dots\wedge e_{i_p}$. 
We also denote 
\[
\partial e_I=\sum_{\alpha=1}^p (-1)^{\alpha-1}e_{i_1}\wedge\dots\wedge
\widehat{e_{i_\alpha}}\wedge\dots\wedge e_{i_p}. 
\]
The Orlik-Solomon algebra is defined as the quotient of 
$\wedge E$ by the following Orlik-Solomon ideal. 
\begin{definition}
(1) A subset $I\subset [n]$ is called dependent if $H_I\neq\emptyset$ 
and $\codim H_I<\#I$. 

(2) The Orlik-Solomon ideal $\mathcal{I}(\A)$ of $\wedge E$ is 
the ideal generated by the following two types of elements. 
\[
\mathcal{I}(\A)=
\left\langle
e_I\mid I\subset [n],  H_I=\emptyset
\right\rangle
+
\left\langle
\partial e_I\mid I\subset [n]\mbox{ is dependent}
\right\rangle. 
\]

(3) The Orlik-Solomon algebra $\OS^\bullet(\A)$ is defined as 
follows. 
\[
\OS^\bullet(\A)=\wedge E/\mathcal{I}(\A). 
\]
\end{definition}
Since the Orlik-Solomon ideal is generated by homogeneous 
elements, $\OS^\bullet(\A)$ is a graded algebra. 
The Hilbert series of the Orlik-Solomon algebra 
$\Hilb(\OS^\bullet(\A), t)=\sum_{k}\rank_\Z\OS^k(\A)\cdot t^k$ 
is related to the characteristic polynomial by the following formula. 
\[
\Hilb(\OS^\bullet(\A), t)=(-t)^\ell\chi(\A, -\frac{1}{t}).
\]

\begin{example}
Let $\A$ be as in Example \ref{ex:A3}. 
Then, the Orlik-Solomon ideal is 
\[
\mathcal{I}(\A)=\langle e_{23}, e_{45}, e_{24}-e_{14}+e_{12}, 
e_{35}-e_{15}+e_{13}\rangle. 
\]
It is easily seen that $1\in\OS^0(\A), e_1, e_2, e_3, e_4, e_5\in 
\OS^1(\A)$ and 
\[
e_{25}, e_{34}, e_{12}, e_{24}, e_{13}, e_{35}\in\OS^2(\A)
\]
form a basis of $\OS^\bullet(\A)$. 
\end{example}

\subsection{Cohomology ring of the complement of complex arrangements} 

Let $\A=\{H_1, \dots, H_n\}$ be an arrangement in $V=\C^\ell$. 
Fix a defining equation $\alpha_i$ for each hyperplane. 
Then 
\[
\omega_i=\frac{d\log\alpha_i}{2\pi\sqrt{-1}}=
\frac{1}{2\pi\sqrt{-1}}\frac{d\alpha_i}{\alpha_i}
\]
is clearly a closed $1$-form on $M(\A)$ and gives an 
element in $H^1(M(\A), \Z)$. 

\begin{proposition}
\label{prop:OSPoin}
(Orlik-Solomon \cite{orl-sol}) 

(1) The map $e_i\longmapsto\omega_i$ 
induces an isomorphism 
$\OS^\bullet(\A)\stackrel{\simeq}{\longrightarrow}
H^\bullet(M(\A), \Z)$ of graded $\Z$-algebras. 

(2) The Poincar\'e polynomial of the complement $M(\A)$ is 
equal to $(-t)^\ell\chi(\A, -1/t)$. In other words, let 
$b_i=b_i(M(\A))$, then $\chi(\A, t)=t^\ell-b_1t^{\ell-1}+\cdots
+(-1)^\ell b_\ell$. 

\end{proposition}
Combining Proposition \ref{prop:OSPoin} and 
Proposition \ref{prop:chicount}, we have the following. 
\begin{corollary}
\label{prop:propM}
Let $\A=\{H_1, \dots, H_n\}$ be an essential 
arrangement in $\R^\ell$. Then 
\[
\begin{split}
\#\ch(\A)&=\sum_{i=0}^\ell b_i(M(\A)), \\
\#\bch(\A)&=\sum_{i=0}^\ell (-1)^{\ell-i} b_i(M(\A)), 
\end{split}
\]
\end{corollary}

The defining equations $\alpha_1, \dots, \alpha_n$ give an 
affine linear map 
\begin{equation}
\label{eq:afflin}
\bm{\alpha}=(\alpha_1, \dots, \alpha_n): \C^\ell\longrightarrow\C^n
\end{equation}
which induces the map of complements 
\begin{equation}
\label{eq:tuple}
\bm{\alpha}|_{M(\A)}: M(\A)\longrightarrow(\C^\times)^n. 
\end{equation}
Note that $\A$ is essential if and only if $\bm{\alpha}$ is injective. 
The description of the cohomology ring in 
Proposition \ref{prop:OSPoin} implies the surjectivity 
of the pull-back of this map. More generally, we can prove 
injectivity (resp. surjectivity) of the induced map on the homology 
(resp. cohomology) for inclusions of affine spaces as follows. 
\begin{proposition}
\label{prop:inj}
Let 
$\A=\{H_1, \dots, H_n\}$ be an arrangement in $V=\C^\ell$, 
and $M=M(\A)$ be its complement. 
Let $F\subset V$ be a $k$-dimensional affine subspace which 
is not contained in $H_i$ for any $1\leq i\leq n$. Then 
the inclusion $\iota: F\cap M\hookrightarrow M$ induces 
the surjection $\iota^*: H^j(M, \Z)\rightarrow H^j(M\cap F, \Z)$ 
of cohomology groups and the injection 
$\iota_*: H_j(F\cap M, \Z)\hookrightarrow H_j(M, \Z)$ of 
homology groups ($j\geq 0$). If $F$ is generic with respect to 
$\A$ (and the hyperplane at infinity), 
it is isomorphic for $j\leq k$. 
\end{proposition}

\section{Chambers and adjacency relation}

Let $\A=\{H_1, \dots, H_n\}$ be an arrangement of 
affine hyperplanes in $\R^\ell$. For simplicity, we assume 
$\A$ is essential. In this section, we 
summarize the basic notions related to real stratifications. 

First let 
\[
\scF: F^0\subset F^1\subset\cdots\subset F^{\ell-1}
\subset F^\ell=\R^\ell 
\]
be a generic flag. Associated with the flag, let us define 
\begin{equation}
\label{eq:flagch}
\ch_{\scF}^k(\A)=\{C\in\ch(\A)\mid C\cap F^k\neq\emptyset, 
C\cap F^{k-1}=\emptyset\} 
\end{equation}
for $k=0, \dots, \ell$, where $F^{-1}=\emptyset$. Then we have 
the following. 
\begin{proposition}
\label{prop:chk}
(\cite[Proposition 2.3.2]{yos-lef}) 
\begin{equation}
\label{eq:232}
\#\ch_\scF^k(\A)=b_k(M(\A)). 
\end{equation}
(Figure \ref{fig:flag}). 
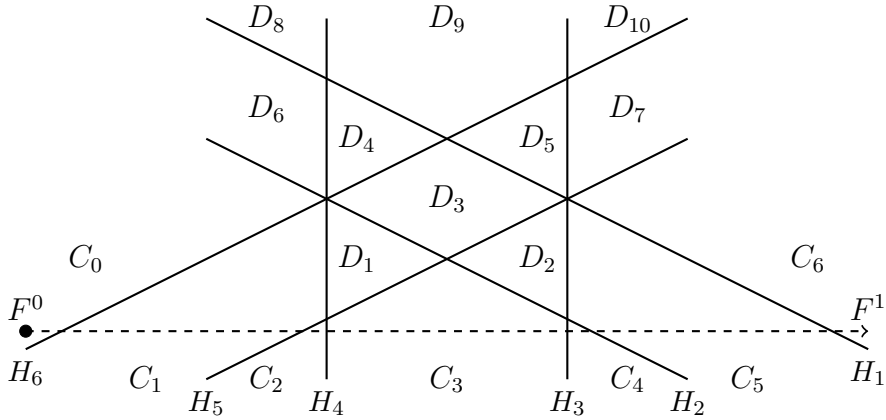
\begin{figure}[htbp]
\centering
\begin{tikzpicture}[scale=0.8]


\draw[thick, black] (0,6)  -- +(11,-5.5) node[below] {\small $H_1$}; 
\draw[thick, black] (0,4) -- +(8,-4) node[below] {\small $H_2$} ; 
\draw[thick, black] (6,0) node[below] {\small $H_3$} -- +(0,6); 
\draw[thick, black] (2,0) node[below] {\small $H_4$}-- +(0,6); 
\draw[thick, black] (0,0) node[below] {\small $H_5$} -- +(8,4); 
\draw[thick, black] (-3,0.5) node[below] {\small $H_6$} -- +(11,5.5); 

\filldraw[fill=black, draw=black] (-3,0.8) node[above] {$F^0$} circle (0.1);
\draw[thick, dashed, ->] (-3,0.8) -- (11,0.8) node[above] {$F^1$} ;

\draw (-2,2) node {$C_0$}; 
\draw (-1,0) node {$C_1$}; 
\draw (1,0) node {$C_2$}; 
\draw (4,0) node {$C_3$}; 
\draw (7,0) node {$C_4$}; 
\draw (9,0) node {$C_5$}; 
\draw (10,2) node {$C_6$}; 

\draw (2.5,2) node {$D_1$}; 
\draw (5.5,2) node {$D_2$}; 
\draw (4,3) node {$D_3$}; 
\draw (2.5,4) node {$D_4$}; 
\draw (5.5,4) node {$D_5$}; 
\draw (1,4.5) node {$D_6$}; 
\draw (7,4.5) node {$D_7$}; 
\draw (1,6) node {$D_8$}; 
\draw (4,6) node {$D_{9}$}; 
\draw (7,6) node {$D_{10}$}; 

\end{tikzpicture}
\caption{A flag and associate sets of chambers: 
$\ch_\scF^0(\A)=\{C_0\}, 
\ch_\scF^1(\A)=\{C_1, \dots, C_6\}, 
\ch_\scF^2(\A)=\{D_1, \dots, D_{10}\}$} 
\label{fig:flag}
\end{figure}
\end{proposition}
\begin{proof}
The proof is by induction on $\ell$. If $\ell=1$, clearly we have 
$\#\ch_\scF^0(\A)=1, \#\ch_\scF^1(\A)=n-1$, which 
are equal to $b_0(M(\A))$ and $b_1(M(\A))$, respectively. 
Suppose $\ell>1$. Since the subspace $F^{\ell-1}$ is generic, 
$L(\A\cap F^{\ell-1})$ is isomorphic to 
\[
L(\A)_{\leq\ell-1}=\{X\in L(\A)\mid \codim X\leq\ell-1\}, 
\]
and we have $b_k(M(\A))=b_k(M(\A\cap F^{\ell-1}))$ for 
$0\leq k\leq\ell-1$. 
Then by Proposition \ref{prop:propM}, we have 
\[
\#\{C\in\ch(\A)\mid C\cap F^{\ell-1}\neq\emptyset\}=
\sum_{k=0}^{\ell-1}b_k(M(\A)). 
\]
Hence $\#\ch_\scF^\ell(\A)=b_\ell(M(\A))$. 
\end{proof}

\begin{remark}
\label{rem:flagnearinfty}
For simplicity, in this article, we pose an extra assumption 
on the flag. 
We consider only flags satisfying the following property: 
$F^{k-1}$ does not separates $0$-dimensional 
intersections of the arrangement $\A\cap F^k$. 
In other words, $F^{\ell-1}$ does not separates 
$0$-dimensional intersections of $\A$, $F^{\ell-2}$ 
does not separates $0$-dimensional intersections 
of the induced arrangement $\A\cap F^{\ell-1}$ on 
$F^{\ell-1}$, and so on. 
We call such a flag 
``near to the hyperplane at infinity.'' 
\end{remark}

Each hyperplane $H\in\A$ determines two regions. We choose 
a ``positive side'' for each $H\in\A$. This leads to 
a \emph{sign vector} $\sigma(x)\in\{\pm, 0\}^{\A}$ 
for every point $x\in\R^\ell$. The set of all points $x\in\R^\ell$ having 
the same sign vector forms a \emph{cell}. Every cell is relatively 
open and convex. The set $\R^\ell$ is decomposed into a disjoint 
union of cells. The cells form the \emph{face poset} 
$\scF(\A)$, ordered by the inclusion of their topological closures. 
Note that the maximal cells are the chambers of $\A$. 
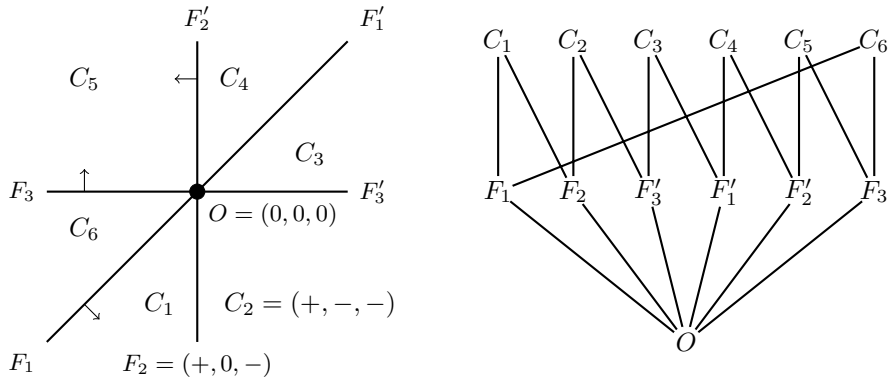
\begin{figure}[htbp]
\centering
\begin{tikzpicture}


\draw[thick] (0,0) node[anchor=north east] {\scriptsize $F_1$} -- ++(4,4)node[anchor=south west] {\scriptsize $F_1'$}; 
\draw[thick] (2,0) node[anchor=north] {\scriptsize $F_2=(+,0,-)$} -- ++(0,4) node[anchor=south] {\scriptsize $F_2'$}; 
\draw[thick] (0,2) node[anchor=east] {\scriptsize $F_3$}-- ++(4,0) node[anchor=west] {\scriptsize $F_3'$}; 

\draw[->] (0.5,0.5) -- ++(0.2, -0.2);
\draw[->] (0.5,2) -- ++(0, 0.3);
\draw[->] (2,3.5) -- ++(-0.3,0);

\filldraw[fill=black, draw=black] (2,2) node[anchor=north west] {\scriptsize $O=(0,0,0)$} circle (0.1); 

\draw (1.5,0.5) node {\footnotesize $C_1$}; 
\draw (3.5,0.5) node {\footnotesize $C_2=(+,-,-)$}; 
\draw (3.5,2.5) node {\footnotesize $C_3$}; 
\draw (2.5,3.5) node {\footnotesize $C_4$}; 
\draw (0.5,3.5) node {\footnotesize $C_5$}; 
\draw (0.5,1.5) node {\footnotesize $C_6$}; 

\coordinate (O) at (8.5,0);

\coordinate (F1) at (6,2);
\coordinate (F2) at (7,2);
\coordinate (F3p) at (8,2);
\coordinate (F1p) at (9,2);
\coordinate (F2p) at (10,2);
\coordinate (F3) at (11,2);

\coordinate (C1) at (6,4);
\coordinate (C2) at (7,4);
\coordinate (C3) at (8,4);
\coordinate (C4) at (9,4);
\coordinate (C5) at (10,4);
\coordinate (C6) at (11,4);

\draw[thick] (O)--(F1);
\draw[thick] (O)--(F2);
\draw[thick] (O)--(F3);
\draw[thick] (O)--(F1p);
\draw[thick] (O)--(F2p);
\draw[thick] (O)--(F3p);
\draw[thick] (F1)--(C1)--(F2)--(C2)--(F3p)--(C3)--(F1p)--(C4)--(F2p)--(C5)--(F3)--(C6)--cycle;

\filldraw[fill=white, draw=white]  (O) node {\footnotesize $O$} circle (0.2); 
\filldraw[fill=white, draw=white]  (F1) node {\footnotesize $F_1$} circle (0.2); 
\filldraw[fill=white, draw=white]  (F2) node {\footnotesize $F_2$} circle (0.2); 
\filldraw[fill=white, draw=white]  (F3p) node {\footnotesize $F_3'$} circle (0.2); 
\filldraw[fill=white, draw=white]  (F1p) node {\footnotesize $F_1'$} circle (0.2); 
\filldraw[fill=white, draw=white]  (F2p) node {\footnotesize $F_2'$} circle (0.2); 
\filldraw[fill=white, draw=white]  (F3) node {\footnotesize $F_3$} circle (0.2); 
\filldraw[fill=white, draw=white]  (C1) node {\footnotesize $C_1$} circle (0.2); 
\filldraw[fill=white, draw=white]  (C2) node {\footnotesize $C_2$} circle (0.2); 
\filldraw[fill=white, draw=white]  (C3) node {\footnotesize $C_3$} circle (0.2); 
\filldraw[fill=white, draw=white]  (C4) node {\footnotesize $C_4$} circle (0.2); 
\filldraw[fill=white, draw=white]  (C5) node {\footnotesize $C_5$} circle (0.2); 
\filldraw[fill=white, draw=white]  (C6) node {\footnotesize $C_6$} circle (0.2); 

\end{tikzpicture}
\caption{Sign vectors and the face poset} 
\label{fig:faceposet}
\end{figure}

Next, we define the adjacency graph of $\A$. 
\begin{definition}
For two chambers $C_1, C_2\in\ch(\A)$, 
let $\Sep(C_1, C_2)=\{H\in\A\mid
\mbox{$H$ separates $C_1$ and $C_2$}\}$ be 
the set of separating hyperplanes. $C_1$ and $C_2$ are 
said to be adjacent if $\#\Sep(C_1, C_2)=1$. 
The \emph{adjacency graph} of $\A$ is the graph $\Gamma(\A)$ 
with the vertex set $\ch(\A)$ and the edges corresponding to 
adjacent pairs of chambers (Figure \ref{fig:adjacencyg}). 
\end{definition}

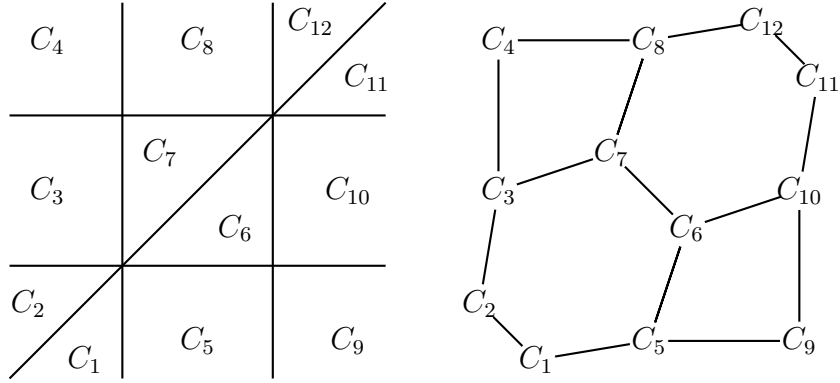
\begin{figure}[htbp]
\centering
\begin{tikzpicture}


\coordinate (C1) at (1.5,0.75);
\coordinate (C2) at (0.75,1.5);
\coordinate (C3) at (1,3);
\coordinate (C4) at (1,5);

\coordinate (C5) at (3,1);
\coordinate (C6) at (3.5,2.5);
\coordinate (C7) at (2.5,3.5);
\coordinate (C8) at (3,5);

\coordinate (C9) at (5,1);
\coordinate (C10) at (5,3);
\coordinate (C11) at (5.25,4.5);
\coordinate (C12) at (4.5,5.25);

\draw (C1) node {$C_1$}; 
\draw (C2) node {$C_2$}; 
\draw (C3) node {$C_3$}; 
\draw (C4) node {$C_4$}; 
\draw (C5) node {$C_5$}; 
\draw (C6) node {$C_6$}; 
\draw (C7) node {$C_7$}; 
\draw (C8) node {$C_8$}; 
\draw (C9) node {$C_9$}; 
\draw (C10) node {$C_{10}$}; 
\draw (C11) node {$C_{11}$}; 
\draw (C12) node {$C_{12}$}; 

\draw[thick] (0.5,0.5) -- ++(5,5); 
\draw[thick] (2,0.5) -- ++(0,5); 
\draw[thick] (4,0.5) -- ++(0,5); 
\draw[thick] (0.5,2) -- ++(5,0); 
\draw[thick] (0.5,4) -- ++(5,0); 


\coordinate (B1) at (7.5,0.75);
\coordinate (B2) at (6.75,1.5);
\coordinate (B3) at (7,3);
\coordinate (B4) at (7,5);

\coordinate (B5) at (9,1);
\coordinate (B6) at (9.5,2.5);
\coordinate (B7) at (8.5,3.5);
\coordinate (B8) at (9,5);

\coordinate (B9) at (11,1);
\coordinate (B10) at (11,3);
\coordinate (B11) at (11.25,4.5);
\coordinate (B12) at (10.5,5.25);

\draw[thick] (B1)--(B2)--(B3)--(B4)--(B8)--(B12)--(B11)--(B10)--(B9)--(B5)--(B1); 
\draw[thick] (B3)--(B7)--(B8)--(B7)--(B6)--(B5)--(B6)--(B10);

\filldraw[fill=white, draw=white]  (B1) node {$C_1$} circle (0.25); 
\filldraw[fill=white, draw=white]  (B2) node {$C_2$} circle (0.25); 
\filldraw[fill=white, draw=white]  (B3) node {$C_3$} circle (0.25); 
\filldraw[fill=white, draw=white]  (B4) node {$C_4$} circle (0.25); 
\filldraw[fill=white, draw=white]  (B5) node {$C_5$} circle (0.25); 
\filldraw[fill=white, draw=white]  (B6) node {$C_6$} circle (0.25); 
\filldraw[fill=white, draw=white]  (B7) node {$C_7$} circle (0.25); 
\filldraw[fill=white, draw=white]  (B8) node {$C_8$} circle (0.25); 
\filldraw[fill=white, draw=white]  (B9) node {$C_9$} circle (0.25); 
\filldraw[fill=white, draw=white]  (B10) node {$C_{10}$} circle (0.25); 
\filldraw[fill=white, draw=white]  (B11) node {$C_{11}$} circle (0.25); 
\filldraw[fill=white, draw=white]  (B12) node {$C_{12}$} circle (0.25); 

\end{tikzpicture}
\caption{Adjacency graph} 
\label{fig:adjacencyg}
\end{figure}

\section{Complexified complement} 
\label{sec:cpxf}
Let $x\in\R^\ell$ and $v\in T_x\R^\ell$ be a tangent vector 
at $x$. Since we can canonically identify 
$T_x\R^\ell\simeq\R^\ell$, we can regard $v$ as an element 
of $\R^\ell$. 
We identify the total space of the tangent bundle $T\R^\ell$ 
with $\C^\ell$ by the map: 
$v\in T_{x}\R^\ell\longmapsto x+\sqrt{-1}\cdot v\in\C^\ell$ 
(Figure \ref{fig:cpxf}). 

Let $H\subset V=\R^\ell$ be a hyperplane defined by a linear 
form $\alpha\in V^*$. Then the complex point 
$x+\sqrt{-1}\cdot v$ is contained in the 
complexified hyperplane $H_\C=H\otimes\C$ if and only if 
\[
\alpha(x+\sqrt{-1}\cdot v)=
\alpha(x)+\sqrt{-1}\cdot\alpha(v)=0, 
\]
which is equivalent to both the real part ($x$) and the imaginary part 
($v$) being contained in $H$. 

Let $\A=\{H_1, \dots, H_n\}$ be a central arrangement in 
$V=\R^\ell$. Using the above description, one can describe the 
complexified complement $M(\A)=\C^\ell\setminus
\bigcup_{i=1}^n H_\C$ as follows (see also Figure \ref{fig:cpxf}). 
\[
\{x+\sqrt{-1}\cdot v\mid \mbox{ if 
$x\in H_i$ for some $1\leq i\leq n$, then }
v\notin T_{x}H_i\}. 
\]
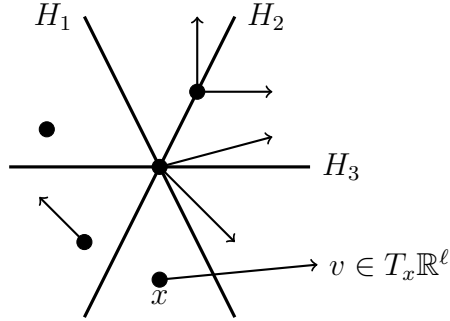
\begin{figure}[htbp]
\centering
\begin{tikzpicture}


\draw[very thick] (0,2)--(4,2) node[right] {$H_3$};
\draw[very thick] (1,0)--(3,4) node[right] {$H_2$};
\draw[very thick] (3,0)--(1,4) node[left] {$H_1$};

\filldraw[fill=black, draw=black] (0.5, 2.5) circle (0.1);

\filldraw[fill=black, draw=black] (1,1) circle (0.1);
\draw[thick, ->] (1,1)--++(-0.6,0.6);

\filldraw[fill=black, draw=black] (2,0.5) node[below] {${x}$} circle (0.1);
\draw[thick, ->] (2,0.5)--++(2.1,0.2) node[right] {${v}\in T_{x}\R^\ell$};

\filldraw[fill=black, draw=black] (2,2) circle (0.1);
\draw[thick, ->] (2,2)--++(1.5,0.4);
\draw[thick, ->] (2,2)--++(1,-1);

\filldraw[fill=black, draw=black] (2.5,3) circle (0.1);
\draw[thick, ->] (2.5,3)--++(1,0);
\draw[thick, ->] (2.5,3)--++(0,1);

\end{tikzpicture}
\caption{Vectors in the complexified complement.} 
\label{fig:cpxf}
\end{figure}

\begin{example}
Let $S_1$ be the set of tangent vectors of a circle in $\R^2$ 
(left of Figure \ref{fig:embcirc}), and 
$S_2$ be the set of radial vectors of a circle in $\R^2$ 
(right of Figure \ref{fig:embcirc}). Then $S_1$ gives a nontrivial 
cycle in the complement of the complexified line $H$, while 
$S_2$ gives a trivial cycle. 
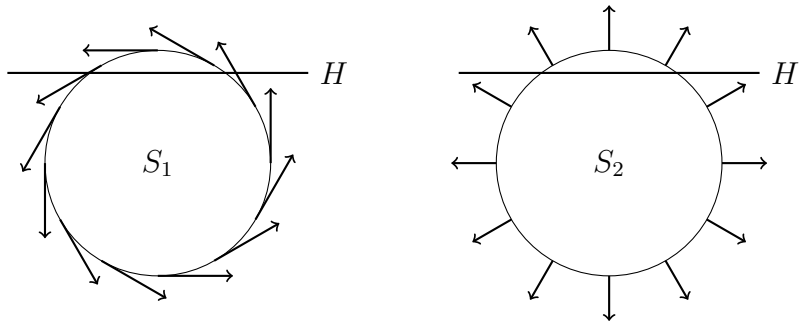
\begin{figure}[htbp]
\centering
\begin{tikzpicture}


\draw[thin] (1.5,1.5) node {$S_1$} circle (1.5);
\draw[thick, ->] (1.5,1.5) ++(0:1.5) -- ++(90:1);
\draw[thick, ->] (1.5,1.5) ++(30:1.5) -- ++(120:1);
\draw[thick, ->] (1.5,1.5) ++(60:1.5) -- ++(150:1);
\draw[thick, ->] (1.5,1.5) ++(90:1.5) -- ++(180:1);
\draw[thick, ->] (1.5,1.5) ++(120:1.5) -- ++(210:1);
\draw[thick, ->] (1.5,1.5) ++(150:1.5) -- ++(240:1);
\draw[thick, ->] (1.5,1.5) ++(180:1.5) -- ++(270:1);
\draw[thick, ->] (1.5,1.5) ++(210:1.5) -- ++(300:1);
\draw[thick, ->] (1.5,1.5) ++(240:1.5) -- ++(330:1);
\draw[thick, ->] (1.5,1.5) ++(270:1.5) -- ++(0:1);
\draw[thick, ->] (1.5,1.5) ++(300:1.5) -- ++(30:1);
\draw[thick, ->] (1.5,1.5) ++(330:1.5) -- ++(60:1);
\draw[thick] (-0.5,2.7) -- ++(4,0) node[right] {$H$};

\draw[thin] (7.5,1.5) node {$S_2$}  circle (1.5);
\draw[thick, ->] (7.5,1.5) ++(0:1.5) -- ++(0:0.6);
\draw[thick, ->] (7.5,1.5) ++(30:1.5) -- ++(30:0.6);
\draw[thick, ->] (7.5,1.5) ++(60:1.5) -- ++(60:0.6);
\draw[thick, ->] (7.5,1.5) ++(90:1.5) -- ++(90:0.6);
\draw[thick, ->] (7.5,1.5) ++(120:1.5) -- ++(120:0.6);
\draw[thick, ->] (7.5,1.5) ++(150:1.5) -- ++(150:0.6);
\draw[thick, ->] (7.5,1.5) ++(180:1.5) -- ++(180:0.6);
\draw[thick, ->] (7.5,1.5) ++(210:1.5) -- ++(210:0.6);
\draw[thick, ->] (7.5,1.5) ++(240:1.5) -- ++(240:0.6);
\draw[thick, ->] (7.5,1.5) ++(270:1.5) -- ++(270:0.6);
\draw[thick, ->] (7.5,1.5) ++(300:1.5) -- ++(300:0.6);
\draw[thick, ->] (7.5,1.5) ++(330:1.5) -- ++(330:0.6);
\draw[thick] (5.5,2.7) -- ++(4,0) node[right] {$H$};

\end{tikzpicture}
\caption{Linked and unlinked embedded circles} 
\label{fig:embcirc}
\end{figure}
\end{example}

\section{Fundamental groups} 

\subsection{Commutator relation}

We start with two lines $\A=\{H_1, H_2\}$ in $\R^2$ 
(left of Figure \ref{fig:2lines}). 

\begin{figure}[htbp]
\centering
\begin{tikzpicture}


\draw[very thick] (0,0)node[below] {$H_2$}--(3,3) ;
\draw[very thick] (3,0)node[below] {$H_1$}--(0,3) ;
\draw[thick, dashed] (-0.5,0.5) node[left] {$F$} -- ++(4,0); 

\draw[thick] (5,0)--++(5,0)--++(0,3)--++(-5,0) node[anchor=north west] {\footnotesize $F\otimes\C$}--cycle;

\draw[thick] (6.5, 1.5) circle (0.07); 
\draw[thick] (8.5, 1.5) circle (0.07); 
\draw (6.5,2.1) node {\footnotesize $H_2\cap F$};
\draw (8.5,2.1) node {\footnotesize $H_1\cap F$};

\draw[thick] (6.5, 1.5) circle (0.35); 
\draw[thick, ->] (6.5,1.85)--(6.49,1.85);
\draw (7.1,1.5) node {\footnotesize $\gamma_2$};
\draw[thick] (8.5, 1.5) circle (0.35); 
\draw[thick, ->] (8.5,1.85)--(8.49,1.85);
\draw (9.1,1.5) node {\footnotesize $\gamma_1$};

\draw[thick] (6.5,1.15)--(7.5,0)--(8.5,1.15);
\filldraw[fill=black, draw=black] (7.5,0) circle (0.08);

\end{tikzpicture}
\caption{Two lines $\A=\{H_1, H_2\}$.} 
\label{fig:2lines}
\end{figure}
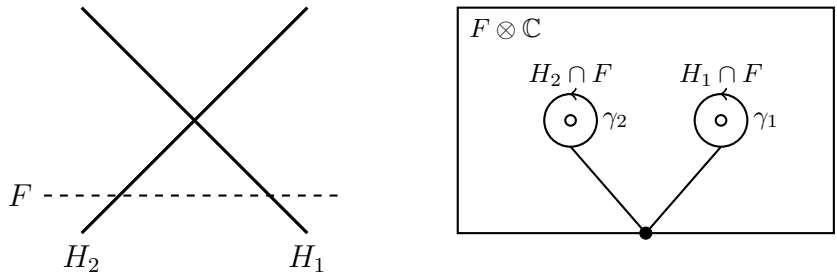
Then the complement $M(\A)$ is homeomorphic to 
\[
\{(z_1, z_2)\in\C^2\mid z_1\neq 0, z_2\neq 0\}=
\C^\times\times\C^\times. 
\]
Hence, $M(\A)$ is homotopy equivalent to the $2$-torus $(S^1)^2$. 
The fundamental group is abelian and isomorphic to $\Z^2$. 
To describe the generators, we first fix a generic line $F$ 
and take meridians as in the right of Figure \ref{fig:2lines}. 
In particular, $\gamma_1$ and $\gamma_2$ are commutative. 

Next, we consider the line arrangement as in Figure \ref{fig:mlines}. 
This arrangement can be considered as a small perturbation 
of the above $\A=\{H_1, H_2\}$. Hence, the commutativity 
of $\gamma_1$ and $\gamma_2$ is preserved. 
\begin{figure}[htbp]
\centering
\begin{tikzpicture}


\draw[thick] (-0.2,0)--(3.4,3) ;
\draw[thick] (0,0)--(3,3) ;
\draw[thick] (0.15,0)--(3.15,3) ;

\draw[thick] (3.2,0)--(-0.3,3) ;
\draw[thick] (2.7,0)--(0.2,3) ;
\draw[thick] (3,0)--(0,3) ;
\draw[thick] (3,0)--(0,3) ;

\draw[thick, dashed] (-0.5,0.5) node[left] {$F$} -- ++(4,0); 

\draw[thick] (5,0)--++(5,0)--++(0,3)--++(-5,0) node[anchor=north west] {\footnotesize $F\otimes\C$}--cycle;

\draw[thick] (6.2, 1.5) circle (0.06); 
\draw[thick] (6.5, 1.5) circle (0.06); 
\draw[thick] (6.8, 1.5) circle (0.06); 

\draw[thick] (8.2, 1.5) circle (0.06); 
\draw[thick] (8.5, 1.5) circle (0.06); 
\draw[thick] (8.8, 1.5) circle (0.06); 


\draw[thick] (6.5, 1.5) circle (0.5); 
\draw[thick, ->] (6.5,2)--(6.49,2) node[above] {\footnotesize $\beta$};
\draw[thick] (8.5, 1.5) circle (0.5); 
\draw[thick, ->] (8.5,2)--(8.49,2) node [above] {\footnotesize $\alpha$};

\draw[thick] (6.5,1)--(7.5,0)--(8.5,1);
\filldraw[fill=black, draw=black] (7.5,0) circle (0.08);

\end{tikzpicture}
\caption{Commutative loops 
$\alpha\beta=\beta\alpha$} 
\label{fig:mlines}
\end{figure}
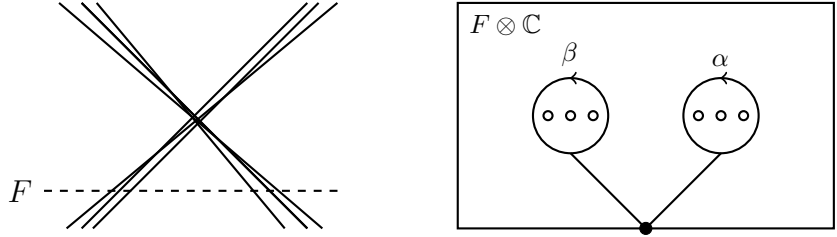
More systematically, we have the following. 

\begin{proposition}
Let $\A=\{H_1, \dots, H_n\}$ be an arrangement of 
$n$ lines intersecting at a point. Let $\gamma_1, \dots, \gamma_n$ 
be meridians as in Figure \ref{fig:cyclicrel}. Then for each $1\leq k< n$, 
$\gamma_1\gamma_2\dots\gamma_k$ and 
$\gamma_{k+1}\dots\gamma_n$ are commutative. 
In other words, we have the following cyclic relation. 
\begin{equation}
\label{eq:cyc}
\gamma_1\gamma_2\cdots\gamma_n=
\gamma_2\cdots\gamma_n\gamma_1=
\gamma_3\cdots\gamma_n\gamma_1\gamma_2=\cdots=
\gamma_n\gamma_1\cdots\gamma_{n-1}
\end{equation}
\begin{figure}[htbp]
\centering
\begin{tikzpicture}


\draw[thick] (3.7,0) node[below] {$H_1$} --(-0.7,3) ;
\draw[thick] (3,0) node[below] {$H_2$} --(0,3) ;
\draw[thick] (1.4,0) node[below] {$H_k$} --(1.6,3) ;
\draw[thick] (0,0) node[below] {$H_n$} --(3,3) ;

\draw (2.2,-0.35) node {$\dots$};
\draw (0.7,-0.35) node {$\dots$};

\draw[thick, dashed] (-0.2,0.5) node[left] {$F$} -- ++(4,0); 

\draw[thick] (5,0)--++(5,0)--++(0,3)--++(-5,0) node[anchor=north west] {\footnotesize $F\otimes\C$}--cycle;

\coordinate (P1) at (9,1.5); 
\coordinate (P2) at (8,1.5); 
\coordinate (Pn) at (6,1.5); 

\coordinate (Q1) at (9,1.85); 
\coordinate (Q2) at (8,1.85); 
\coordinate (Qn) at (6,1.85); 

\coordinate (R1) at (9,1.15); 
\coordinate (R2) at (8,1.15); 
\coordinate (Rn) at (6,1.15); 

\coordinate (B) at (7.5,0); 

\draw[thick] (P1) circle (0.06); 
\draw[thick] (P2) circle (0.06); 
\draw[thick] (Pn) circle (0.06); 

\draw[thick] (P1) circle (0.35); 
\draw[thick] (P2) circle (0.35); 
\draw[thick] (Pn) circle (0.35); 

\draw[thick, ->] (Q1) -- ++(-0.01,0) node[above] {$\gamma_1$}; 
\draw[thick, ->] (Q2) -- ++(-0.01,0) node[above] {$\gamma_2$}; 
\draw[thick, ->] (Qn) -- ++(-0.01,0) node[above] {$\gamma_n$}; 

\filldraw[fill=black, draw=black] (B) circle (0.08);

\draw[thick] (Rn)--(B)--(R2)--(B)--(R1);

\draw (7, 1.5) node {$\cdots$}; 

\end{tikzpicture}
\caption{Commutative relation $[\gamma_1\cdots\gamma_k, 
\gamma_{k+1}\cdots\gamma_n]=1$} 
\label{fig:cyclicrel}
\end{figure}
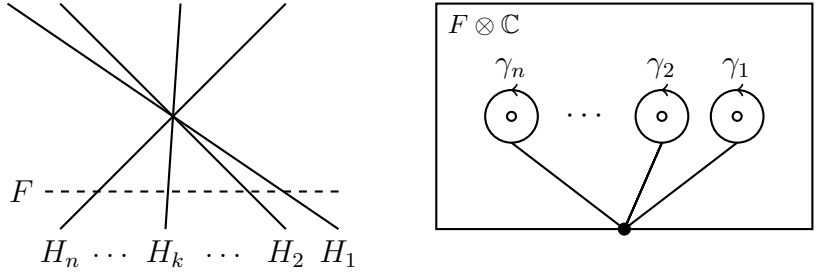
\end{proposition}
Roughly speaking, the cyclic relations (\ref{eq:cyc}) generate all 
relations in the fundamental group of $M(\A)$ for any 
arrangement. 

\subsection{Presentations}
\label{sec:pres}

Let $\A=\{H_1, \dots, H_n\}$ be an affine arrangement in 
$\R^\ell$. By a version of the Lefschetz hyperplane section theorem 
\cite{le-ham}, if $F$ is a generic affine subspace of complex dimension 
$2$, the inclusion $i: F\cap M(\A)\hookrightarrow M(\A)$ induces 
an isomorphism 
\[
i_*: 
\pi_1(F\cap M(\A))\stackrel{\simeq}{\longrightarrow}\pi_1(M(\A)). 
\]
Thus, when we consider the fundamental groups, 
it is enough to consider line arrangements in $\R^2$. 
There are many ways to give presentations of the fundamental group. 
Among others (e.g., \cite{cs-br, ikn, suc-fundam}), we give two of them. 

Now, let $\A=\{H_1, \dots, H_n\}$ be a line arrangement in $V=\R^2$. 
Let $F^0\subset F^1$ be a generic flag in $V$. Let $f(x)$ be a 
polynomial of degree one such that $F^1=\{f=0\}$, and $\alpha_i$ 
($i=1, \dots, n$) be the defining equation of $H_i$. Choose an 
orientation of $F^1$. 
For simplicity, we assume the following (Figure \ref{fig:flag}). 
\begin{itemize}
\item All intersections of $\A$ are contained in the half-space 
$\{f>0\}$. 
\item 
The point $F^0$ is contained in the half-space 
$H_i^{-}:=\{\alpha_i<0\}$ for $i=1, \dots, n$. 
\item 
On $F^1$, the order is 
$F^0<H_n\cap F^1<\cdots<H_2\cap F^1<H_1\cap F^1$, 
with respect to the orientation of $F^1$. 
\end{itemize}
The set of generators is 
\begin{equation}
\label{eq:basicgen}
\{\gamma_X\mid X \mbox{ is a $1$-dimensional face}\}. 
\end{equation}
At each point of multiplicity $m$, we impose the following 
relations (Figure \ref{fig:basicrel}). 
\begin{figure}[htbp]
\centering
\begin{tikzpicture}


\draw[thick] (0,0) node[below] {\footnotesize $X_m$} -- (6,4) node[above] {\footnotesize $Y_m$}; 
\draw[thick] (2.5,0) node[below] {\footnotesize $X_i$} -- (3.5,4) node[above] {\footnotesize $Y_i$}; 
\draw[thick] (5,0) node[below] {\footnotesize $X_2$} -- (1,4) node[above] {\footnotesize $Y_2$}; 
\draw[thick] (6,0) node[below] {\footnotesize $X_1$} -- (0,4) node[above] {\footnotesize $Y_1$} ; 
\draw (1.25,0) node[below] {$\cdots$}; 
\draw (3.75,0) node[below] {$\cdots$}; 
\draw (2.25,4) node[above] {$\cdots$}; 
\draw (4.75,4) node[above] {$\cdots$}; 

\end{tikzpicture}
\caption{$1$-dimensional faces around a point of multiplicity $m$}
\label{fig:basicrel}
\end{figure}
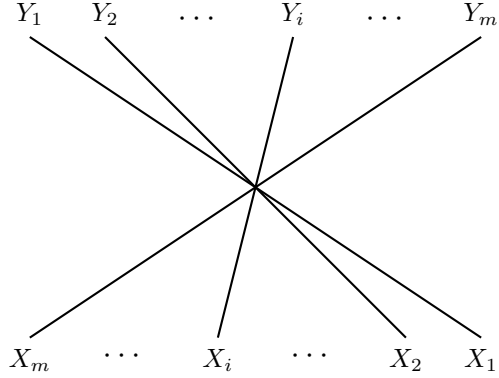
\begin{equation}
\label{eq:basicrel1}
\gamma_{Y_i}=(\gamma_{X_1}\cdots\gamma_{X_{i-1}})\gamma_{X_i}
(\gamma_{X_1}\cdots\gamma_{X_{i-1}})^{-1}, \mbox{ for } i=1, \dots, m, 
\end{equation}
\begin{equation}
\label{eq:basicrel2}
[\gamma_{X_1}\cdots\gamma_{X_i}, \gamma_{X_{i+1}}\cdots\gamma_{X_m}]=1, \mbox{ for } i=1, \dots, m-1. 
\end{equation}

\begin{theorem}
\label{thm:frpres}
(\cite{fal-hom, ran-fun}) For a line arrangement $\A=\{H_1, \dots, 
H_n\}$ in $\R^2$, the fundamental group $\pi_1(M(\A))$ is 
presented by the generators (\ref{eq:basicgen}) and relations 
(\ref{eq:basicrel1}) and (\ref{eq:basicrel2}). 
\end{theorem}
Note that one can reduce the number of generators to $n$ 
by using the relations (\ref{eq:basicrel1}). This gives a 
minimal presentation of the fundamental group. 

Next, we give a minimal presentation directly. The relations 
are indexed by 
\[
\ch_\scF^2(\A)=
\{C\in\ch(\A)\mid C\cap F^1=\emptyset\}. 
\]
To each $D\in\ch_\scF^2(\A)$ we will associate a relation. 
\begin{itemize}
\item
Let $\{i\in [n]\mid \alpha_i(D)<0\}=\{i_1, i_2, \dots, i_p\}$ with 
$i_1<\cdots<i_p$. 
\item
Let $\{i\in [n]\mid \alpha_i(D)>0\}=\{j_1, j_2, \dots, j_q\}$ with 
$j_1<\cdots<j_q$. 
\item 
Define the relation $R(D)$ as follows. 
\[
\gamma_1 \cdots \gamma_n=\gamma_{i_1} \cdots \gamma_{i_p}
\gamma_{j_1}\cdots\gamma_{j_q}
\]
\end{itemize}

\begin{theorem}
\label{thm:minpres}
(\cite{yos-str}) 
The fundamental group $\pi_1(M(\A))$ is 
presented as 
\begin{equation}
\label{eq:minpres}
\langle\gamma_1, \dots, \gamma_n\mid R(D), D\in\ch_{\scF}^2(\A)
\rangle
\end{equation}
\end{theorem}
This can be proved by using Proposition \ref{prop:minstr} in 
\S \ref{sec:min2dim}, see \cite{yos-str} for details. 

\begin{example}
\label{ex:X3pres}
Let $\A=\{H_1, \dots, H_6\}$ 
be the line arrangement as in 
Figure \ref{fig:flag} (see also Figure \ref{fig:chamberrel}). 
The presentation in Theorem \ref{thm:frpres} is reduced to 
the presentation 
\[
\left\langle
\gamma_1, \dots, \gamma_6
\left|
\begin{split}
&[\gamma_2, \gamma_3]=[\gamma_4, \gamma_5]=
[\gamma_2, \gamma_5]=\\
&[\gamma_1 \gamma_3, \gamma_5]=
[\gamma_1, \gamma_3 \gamma_5]=
[\gamma_2 \gamma_4, \gamma_6]=
[\gamma_2, \gamma_4 \gamma_6]=\\
&[\gamma_1, \gamma_6]=[\gamma_1, \gamma_2\gamma_4\gamma_2^{-1}]
=[\gamma_1\gamma_3\gamma_1^{-1}, \gamma_6]=1
\end{split}
\right.
\right\rangle. 
\]
The presentation in Theorem \ref{thm:minpres} is as follows. 
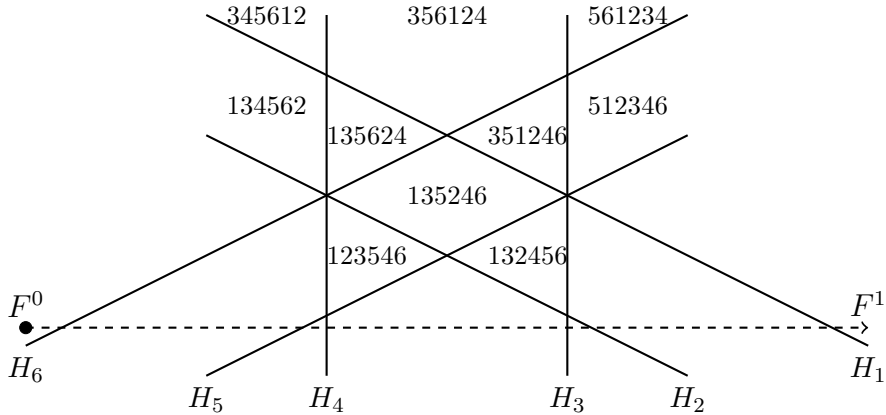
\begin{figure}[htbp]
\centering
\begin{tikzpicture}[scale=0.8]


\draw[thick, black] (0,6)  -- +(11,-5.5) node[below] {\small $H_1$}; 
\draw[thick, black] (0,4) -- +(8,-4) node[below] {\small $H_2$} ; 
\draw[thick, black] (6,0) node[below] {\small $H_3$} -- +(0,6); 
\draw[thick, black] (2,0) node[below] {\small $H_4$}-- +(0,6); 
\draw[thick, black] (0,0) node[below] {\small $H_5$} -- +(8,4); 
\draw[thick, black] (-3,0.5) node[below] {\small $H_6$} -- +(11,5.5); 

\filldraw[fill=black, draw=black] (-3,0.8) node[above] {$F^0$} circle (0.1);
\draw[thick, dashed, ->] (-3,0.8) -- (11,0.8) node[above] {$F^1$} ;

%

\draw (2.5,2) node {\footnotesize \ \ \ $123546$}; 
\draw (5.5,2) node {\footnotesize $132456$\ \ \ \ }; 
\draw (4,3) node {\footnotesize $135246$}; 
\draw (2.5,4) node {\footnotesize $\ \ \ 135624$}; 
\draw (5.5,4) node {\footnotesize $351246$\ \ \ \ }; 
\draw (1,4.5) node {\footnotesize $134562$}; 
\draw (7,4.5) node {\footnotesize $512346$}; 
\draw (1,6) node {\footnotesize $345612$}; 
\draw (4,6) node {\footnotesize $356124$}; 
\draw (7,6) node {\footnotesize $561234$}; 

\end{tikzpicture}
\caption{Chamber relations} 
\label{fig:chamberrel}
\end{figure}
\[
\left\langle
\gamma_1, \dots, \gamma_6
\left|
\begin{array}{l}
123456=\\
132456=123546=135246=135624=351246=\\
134562=345612=356124=561234=512346
\end{array}
\right.
\right\rangle
\]
\end{example}

\begin{remark}
Let us compare the fundamental groups of 
an arrangement $\A$ and its generic 
perturbation $\A'$. 
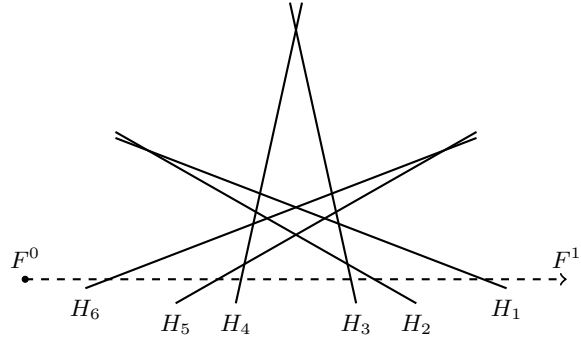
\begin{figure}[htbp]
\centering
\begin{tikzpicture}[scale=0.4]


\draw[thick, black] (11,0.5) node[below] {\scriptsize $H_1$} -- ++(-13,5) ; 
\draw[thick, black] (8,0 ) node[below] {\scriptsize $H_2$} -- ++(-10,5.7)  ; 
\draw[thick, black] (6,0) node[below] {\scriptsize $H_3$} -- ++(-2.2,10); 
\draw[thick, black] (2,0) node[below] {\scriptsize $H_4$}-- ++(2.2,10); 
\draw[thick, black] (0,0) node[below] {\scriptsize $H_5$} -- ++(10,5.7); 
\draw[thick, black] (-3,0.5) node[below] {\scriptsize $H_6$} -- ++(13,5); 

\filldraw[fill=black, draw=black] (-5,0.8) node[above] {\scriptsize$F^0$} circle (0.1);
\draw[thick, dashed, ->] (-5,0.8) -- (13,0.8) node[above] {\scriptsize$F^1$} ;

\end{tikzpicture}
\caption{A generic perturbation of Figure \ref{fig:chamberrel}} 
\label{fig:genper}
\end{figure}
The minimal homogeneous presentation 
of $\pi_1(M(\A'))$ is obtained from that of $\pi_1(M(\A))$ by 
adding several relations determined by the new chambers. 
Since $\pi_1(M(\A'))$ is isomorphic to the free abelian group 
$\Z^n$, we have that $\pi_1(M(\A))$ is a group generated by 
$n$ elements, and adding $\frac{n(n-1)}{2}-b_1$ relations 
makes it the free abelian group of rank $n$. 
\end{remark}

\begin{remark}
CW complexes associated with the minimal presentations in 
Theorem \ref{thm:frpres} and Theorem \ref{thm:minpres} 
are known to be homotopy equivalent to the complement 
$M(\A)$. A natural question is whether or not 
the CW complex associated with any minimal presentation 
of $\pi_1(M(\A))$ is homotopy equivalent to $M(\A)$. 
An affirmative answer to this question yields the following: 
let $\A_1$ and $\A_2$ be line arrangements in $V=\C^2$. 
If $\pi_1(M(\A_1))\simeq \pi_1(M(\A_2))$, then 
$M(\A_1)$ and $M(\A_2)$ are homotopy equivalent. 
\end{remark}

\subsection{Deligne groupoid}

In this section, we construct a category 
$\Gal^+(\A)$ of paths (galleries) of chambers and 
the Deligne groupoid $\Gal(\A)$. 
These were first introduced by 
Deligne \cite{del-kpi1} for simplicial arrangements and 
later Paris formulated for any real arrangements \cite{par-fun}. 

A path of length $k$ is a sequence $(C_0, C_1, \dots, C_k)$ 
of chambers such that $C_{i-1}$ and $C_i$ are adjacent 
($i=1, \dots, k$). 
A path $(C_0, \dots, C_k)$ is said to be 
\emph{geodesic} if $\#\Sep(C_0, C_k)=k$. 
We construct the category of galleries by identifying 
geodesics. 
\begin{definition}
Two paths $(C_0, \dots, C_k)$ and $(D_0, \dots, D_k)$ are 
related by a \emph{geodesic flip} if 
there exist indices $0\leq i<j\leq k$ such that 
\begin{itemize}
\item 
$C_0=D_0, \dots, C_i=D_i$ and $C_j=D_j$, \dots, $C_k=D_k$. 
\item 
both $(C_i, C_{i+1}, \dots, C_j)$ and $(D_i, D_{i+1}, \dots, D_j)$ 
are geodesic. We denote $\sim$ as the equivalence relation 
generated by finitely many geodesic flips. 
\end{itemize}
\end{definition}

\begin{figure}[htbp]
\centering
\begin{tikzpicture}

\draw[thick] (0,0)--++(6.5,6.5);
\draw[thick] (2,0)--++(0,6.5);
\draw[thick] (4,0)--++(0,6.5);
\draw[thick] (0,2)--++(7,0);
\draw[thick] (0,4)--++(7,0);
\draw[thick] (1.5,6.5)--++(5.5,-5.5);

\filldraw[thick, fill=red, draw=red] 
(0,1) circle (0.08) -- 
(1,3) circle (0.05) --
(2.3,3.7) circle (0.05) --
(2.3,4.3) circle (0.05) --
(3.7,5.7) circle (0.05) --
(4.3,5.7) circle (0.05) --
(5.5,5) circle (0.05) --
(5,6) circle (0.08);

\filldraw[thick, fill=olive, draw=olive] 
(0,0.6) circle (0.08) -- 
(1,2.6) circle (0.05) --
(2.6,3.3) circle (0.05) --
(3.4,2.7) circle (0.05) --
(4.6,2.7) circle (0.05) --
(5.4,3.3) circle (0.05) --
(5.8,5) circle (0.05) --
(5.5,6) circle (0.08);

\filldraw[thick, fill=blue, draw=blue] 
(0.3,0.6) circle (0.08) -- 
(1,0.3) circle (0.05) --
(3,0.3) circle (0.05) --
(6,1) circle (0.05) --
(6.5,1.75) circle (0.05) --
(6,3) circle (0.05) --
(6,5) circle (0.05) --
(5.8,6.2) circle (0.08);

\end{tikzpicture}
\caption{Geodesic flips and equivalent paths} 
\label{fig:geod}
\end{figure}
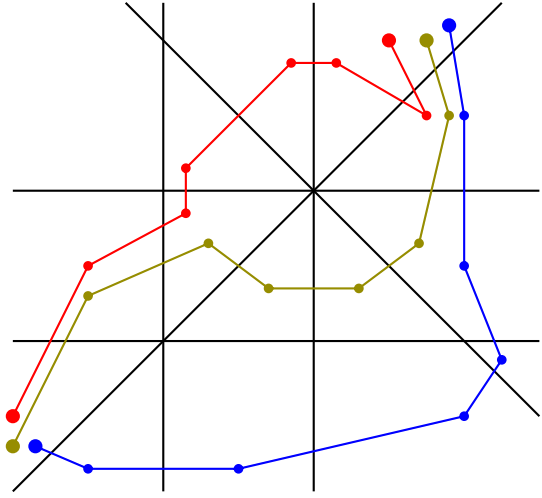

From the adjacency graph, one can define the following 
categories of galleries. 

\begin{definition}
$\Gal^+(\A)$ is the category such that 
\begin{itemize}
\item 
The set of objects is  
$\obj(\Gal^+(\A))=\ch(\A)$. 
\item 
The set of morphisms from $C$ to $C'$ is 
\[
\Hom^+(C, C')=
\left.
\left\{
(C_0, C_1, \dots, C_n)
\left|
\begin{array}{l}
n\geq 0, C_i\in\ch(\A), \\
C_0=C, C_n=C'
\end{array}
\right.\right\}
\right/
\sim,
\] 
where $\sim$ is the equivalence relation as above. 
\end{itemize}
\end{definition}
By inverting all morphisms (taking the free groupoid \cite{mac}), 
we obtain the category 
$\Gal(\A)$. To formulate the inverse of the morphisms, 
we consider signed paths 
\[
(C_0*_{\eps_1}C_1*_{\eps_2}\cdots *_{\eps_k}C_k), 
\]
where $\eps_i\in\{\pm\}$. If all the signs are $+$, we identify 
$(C_0*_{+}C_1*_{+}\cdots *_{+}C_k)$, with the 
path $(C_0, C_1, \dots, C_k)$. If all the signs are $-$, 
we identify 
$(C_0*_{-}C_1*_{-}\cdots *_{-}C_k)$, with the 
path $(C_k, C_{k-1}, \dots, C_1)^{-1}$. 

\begin{definition}
The category $\Gal(\A)$ is the category such that 
\begin{itemize}
\item 
The set of objects is  
$\obj(\Gal(\A))=\ch(\A)$. 
\item 
The set of morphisms from $C$ to $C'$ is 
\[
\Hom(C, C')=
\left.
\left\{
(C_0*_{\eps_1}C_1*_{\eps_2} \dots *_{\eps_n} C_n)
\left|
\begin{array}{l}
n\geq 0, C_i\in\ch(\A), \\
C_0=C, C_n=C'
\end{array}
\right.\right\}
\right/
\sim,
\] 
where $\sim$ is the equivalence relation generated by 
geodesic flips as well as the following: 
\[
\begin{split}
(\cdots *_{\eps_k}C_k*_{\eps_{k+1}}\cdots)
&\sim
(\cdots *_{\eps_k}C_k*_+ D*_- C_k*_{\eps_{k+1}}\cdots)\\
&\sim
(\cdots *_{\eps_k}C_k*_- D*_+ C_k*_{\eps_{k+1}}\cdots), 
\end{split}
\]
where $D$ is a chamber adjacent to $C_k$. 
\end{itemize}
\end{definition}
By the construction, all morphisms of $\Gal(\A)$ are 
isomorphisms, hence $\Gal(\A)$ is a groupoid. 
The Deligne groupoid $\Gal(\A)$ recovers the fundamental 
group of $M(\A)$. 
\begin{theorem}
\label{thm:aut}
Let $C\in\ch(\A)$. Then 
\begin{equation}
\label{eq:autopi1}
\Hom(C, C)\simeq\pi_1(M(\A), p_C), 
\end{equation}
where, $p_C$ is an arbitrary point in $C$. 
\end{theorem}
The strategy of the proof of Theorem \ref{thm:aut} is as follows. 
Choose a point $p_{C'}\in C'$ for each chamber $C'$. We give 
a map $\Hom(C, C)\longmapsto\pi_1(M(\A), p_C)$ by associating 
a curve in $M(\A)$ to a signed paths 
$(C=C_0*_{\varepsilon_1} C_1*_{\varepsilon_2}\cdots 
*_{\varepsilon_k}C_k=C)$ as in Figure \ref{fig:autopi1}. 
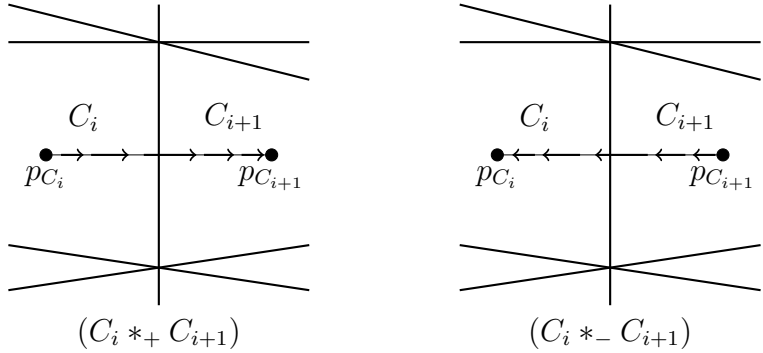
\begin{figure}[htbp]
\centering
\begin{tikzpicture}

\draw[thick] (2,0) node[below] {$(C_{i}*_+ C_{i+1})$} -- ++(0,4);
\draw[thick] (0,3.5) -- ++(4,0);
\draw[thick] (0,4) -- ++(4,-1);
\draw[thick] (0,0.2) -- ++(4,0.6);
\draw[thick] (0,0.8) -- ++(4,-0.6);
\draw (1,2.5) node {$C_{i}$}; 
\draw (3,2.5) node {$C_{i+1}$}; 
\filldraw[fill=black, draw=black] (0.5,2) node[below] {$p_{C_{i}}$} circle (0.08); 
\filldraw[fill=black, draw=black] (3.5,2) node[below] {$p_{C_{i+1}}$} circle (0.08); 
\draw[thick, ->] (0.7,2) -- ++(0.3,0);
\draw[thick, ->] (1.1,2) -- ++(0.5,0);
\draw[thick, ->] (1.8,2) -- ++(0.7,0);
\draw[thick, ->] (2.6,2) -- ++(0.4,0);
\draw[thick, ->] (3.1,2) -- ++(0.3,0);
\draw[very thin] (0.5,2)--++(3,0);

\draw[thick] (8,0) node[below] {$(C_{i}*_- C_{i+1})$} -- ++(0,4);
\draw[thick] (6,3.5) -- ++(4,0);
\draw[thick] (6,4) -- ++(4,-1);
\draw[thick] (6,0.2) -- ++(4,0.6);
\draw[thick] (6,0.8) -- ++(4,-0.6);
\draw (7,2.5) node {$C_{i}$}; 
\draw (9,2.5) node {$C_{i+1}$}; 
\filldraw[fill=black, draw=black] (6.5,2) node[below] {$p_{C_{i}}$} circle (0.08); 
\filldraw[fill=black, draw=black] (9.5,2) node[below] {$p_{C_{i+1}}$} circle (0.08); 
\draw[thick, <-] (6.7,2) -- ++(0.3,0);
\draw[thick, <-] (7.1,2) -- ++(0.5,0);
\draw[thick, <-] (7.8,2) -- ++(0.7,0);
\draw[thick, <-] (8.6,2) -- ++(0.4,0);
\draw[thick, <-] (9.1,2) -- ++(0.3,0);
\draw[very thin] (6.5,2)--++(3,0);

\end{tikzpicture}
\caption{Paths corresponding to $(C_{i}*_+ C_{i+1})$ and 
$C_{i}*_- C_{i+1}$.} 
\label{fig:autopi1}
\end{figure}
This correspondence gives the isomorphism (\ref{eq:autopi1}). 

\begin{remark}
Let $\A$ be a central and simplicial arrangement in $\R^\ell$. 
Deligne \cite{del-kpi1} proved that 
the localization functor $\rho:\Gal^+(\A)\longrightarrow\Gal(\A)$ 
is faithful. This property was one of the crucial steps in his 
proof of the $K(\pi, 1)$ property. 

However, in non simplicial cases, $\rho$ is not faithful 
\cite{yos-lef}. Non faithfulness seems to have a relationship 
with non vanishing of the second homotopy group \cite{sai-can}. 

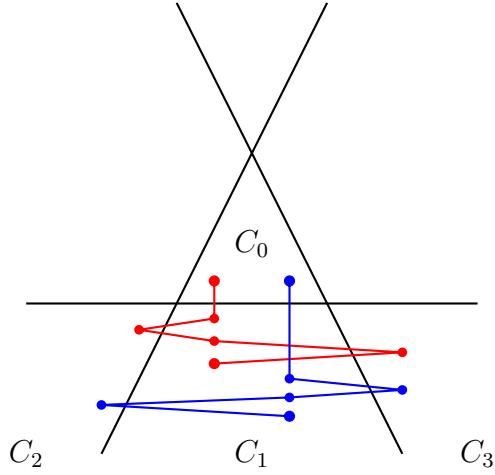
\begin{figure}[htbp]
\centering
\begin{tikzpicture}

\draw[thick] (1,0)--++(3,6);
\draw[thick] (5,0)--++(-3,6);
\draw[thick] (0,2)--++(6,0);

\draw (3,2.8) node {$C_0$}; 
\draw (3,0) node {$C_1$}; 
\draw (0,0) node {$C_2$}; 
\draw (6,0) node {$C_3$}; 

\filldraw[thick, fill=red, draw=red] 
(2.5,2.3) circle (0.06) -- 
(2.5,1.8) circle (0.05) --
(1.5,1.65) circle (0.05) --
(2.5,1.5) circle (0.05) --
(5,1.35) circle (0.05) --
(2.5,1.2)circle (0.06);

\filldraw[thick, fill=blue, draw=blue] 
(3.5,2.3) circle (0.06) -- 
(3.5,1) circle (0.05) --
(5,0.85) circle (0.05) --
(3.5,0.75) circle (0.05) --
(1,0.65) circle (0.05) --
(3.5,0.5)circle (0.06);

\end{tikzpicture}
\caption{$C_0C_1C_2C_1C_3C_1$ (red) and 
$C_0C_1C_3C_1C_2C_1$ (blue) are equivalent.} 
\label{fig:nonfaithful}
\end{figure}
\end{remark}

\section{Homotopy type and Minimality}

\subsection{Salvetti complex}
\label{subsec:sal}

Let $\A=\{H_1, \dots, H_n\}$ be an arrangement in $\R^\ell$. 
Choose a defining equation $\alpha_i$ for each hyperplane $H_i$. 
The hyperplane $H_i$ decomposes $\R^\ell$ into two half-spaces, 
$H_i^+=\{x\in\R^\ell\mid \alpha_i(x)>0\}$ and 
$H_i^-=\{x\in\R^\ell\mid \alpha_i(x)<0\}$. 
Each point $x\in\R^\ell$ determines a sign vector 
\[
(\sign(\alpha_i(x))\mid i=1, \dots, n)\in\{\pm, 0\}^n. 
\]
Let $X, Y:\A\longrightarrow \{\pm, 0\}$ be sign vectors. 
The \emph{composition} $X\circ Y:\A\longrightarrow \{\pm, 0\}$ 
of $X$ and $Y$ is defined by 
\[
(X\circ Y)(H)=
\begin{cases}
X(H), &\ \mbox{ if $X(H)\neq 0$}\\
Y(H), &\ \mbox{ if $X(H)= 0$}.
\end{cases}
\]
Let 
\[
\Sal(\A)=\{(X, C)\in\scF(\A)\times\ch(\A)\mid X\leq C\}. 
\]
Define the partial order on $\Sal(\A)$ as follows 
(Figure \ref{fig:Sal}): 
\[
(X', C')\leq (X, C) \mbox{ if and only if } X\leq X' 
\mbox{ and } X'\circ C=C'. 
\]
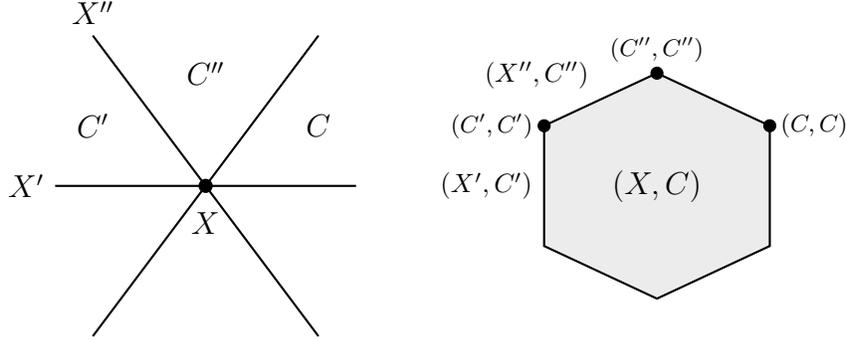
\begin{figure}[htbp]
\centering
\begin{tikzpicture}

\draw[thick] (1,2)--++(4,0);
\draw[thick] (1.5,0)--++(3,4);
\draw[thick] (4.5,0)--++(-3,4);

\filldraw[thick, fill=black, draw=black] (3,2) circle (0.08); 

\draw (3,1.5) node {$X$}; 
\draw (4.5,2.8) node {$C$}; 

\draw (1,2) node[left] {$X'$}; 
\draw (1.5,2.8) node {$C'$}; 

\draw (1.5,4) node[anchor=south] {$X''$}; 
\draw (3,3.5) node {$C''$}; 

\coordinate (Center) at (9,2); 
\coordinate (C1) at (10.5,2.8); 
\coordinate (C2) at (9,3.5); 
\coordinate (C3) at (7.5,2.8); 
\coordinate (C4) at (7.5,1.2); 
\coordinate (C5) at (9,0.5); 
\coordinate (C6) at (10.5,1.2); 

\filldraw[thick, fill=lightgray!30, draw=black] (C1)--(C2)-- node[anchor=south east] {\footnotesize $(X'', C'')$} (C3)-- node[anchor=east] {\footnotesize $(X', C')$} (C4)--(C5)--(C6)--cycle; 

\draw (Center) node {$(X, C)$}; 

\filldraw (C1) circle (0.08) node[anchor=west] {\scriptsize $(C, C)$}; 
\filldraw (C2) circle (0.08) node[anchor=south] {\scriptsize $(C'', C'')$}; 
\filldraw (C3) circle (0.08) node[anchor=east] {\scriptsize $(C', C')$}; 

\end{tikzpicture}
\caption{Salvetti complex} 
\label{fig:Sal}
\end{figure}
The Salvetti complex $\Delta_{\Sal}(\A)$ is 
the regular CW complex that has this $\Sal(\A)$  
as its face poset. 
\begin{theorem}
\label{thm:salcpx}
(\cite{sal-top, par-fun, del-com}) $\Delta_{\Sal}(\A)$ is 
homotopy equivalent to $M(\A)$. 
\end{theorem}
We illustrate the proof of Theorem \ref{thm:salcpx} via nerve lemma. 

\begin{lemma}
\label{lem:conv}
Let $H^+\subset\R^\ell$ be a half-space determined by a 
hyperplane $H$. Let $F\subset\R^\ell$ be an affine space. 
Let $d(x, F)$ be the Euclidean distance from a point $x$ to $F$. 
Then 
\[
U=\{x\in H^+\mid d(x, F)<d(x, H)\}
\]
is an open convex subset. 
\end{lemma}
\begin{proof}
Let $x_1, x_2\in \R^\ell$. 
From the triangle inequality, we can easily prove 
\[
d(tx_1 + (1-t)x_2, F)\leq td(x_1, F)+(1-t)d(x_2, F)
\]
for $0\leq t\leq 1$. On the other hand, if $x_1, x_2\in H^+$, then 
\[
d(tx_1 + (1-t)x_2, H)= td(x_1, H)+(1-t)d(x_2, H). 
\]
Thus $x_1, x_2\in U$ implies $tx_1+(1-t)x_2\in U$. 
\end{proof}
Let $F\in\scF(\A)$ be a face with the sign vector $\sigma(F)=
(\varepsilon_1, \dots, \varepsilon_n)\in\{+, -, 0\}^n$. 
Define the neighborhood 
$V_F$ as the intersection of half-spaces containing $F$, 
\[
V_F=\bigcap_{\varepsilon_i\neq 0}H^{\varepsilon_i}. 
\]
Next, define the subset $V_F$ as 
\[
U_F=\{x\in V_F\mid d(x, F)<d(x, H_i) \mbox{ for $1\leq i\leq n$ with } 
\varepsilon_i\neq 0\}
\]
(Figure \ref{fig:VFUF}). 
Note that for a chamber $C$, $U_C=C$. 
Since $U_F$ is an intersection of 
finitely many open convex subsets as in Lemma \ref{lem:conv}, 
$U_F$ is open and convex. 
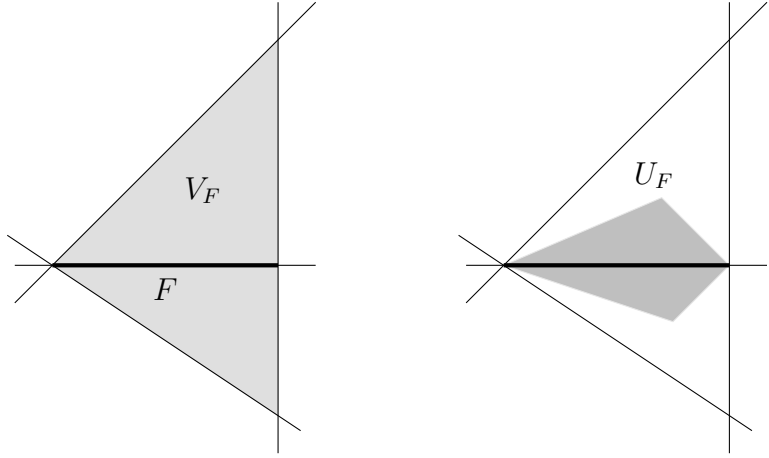
\begin{figure}[htbp]
\centering
\begin{tikzpicture}

\filldraw[fill=lightgray!50, draw=lightgray!50] (1,2)--(4,0)--(4,5)--cycle;
\draw (0.5, 2) -- ++(4,0);
\draw[ultra thick] (1, 2) -- node[below] {$F$} ++(3,0);
\draw (4, -0.5) -- ++(0,6);
\draw (0.5, 1.5) -- ++(4,4);
\draw (0.4, 2.4) -- ++(3.9,-2.6);
\draw (3,3) node {$V_F$}; 

\coordinate (I1) at (9.1,2.9); 
\coordinate (I2) at (9.25,1.25); 

\filldraw[fill=lightgray!100, draw=lightgray!40] (7,2)--(I1)--(10,2)--(I2)--cycle;
\draw (6.5, 2) -- ++(4,0);
\draw[ultra thick] (7, 2) -- ++(3,0);
\draw (10, -0.5) -- ++(0,6);
\draw (6.5, 1.5) -- ++(4,4);
\draw (6.4, 2.4) -- ++(3.9,-2.6);
\draw (9,3.2) node {$U_F$}; 

\end{tikzpicture}
\caption{$V_F$ and $U_F$} 
\label{fig:VFUF}
\end{figure}
The collection 
$\{U_F\mid F\in\scF(\A)\}$ provides an open covering of $\R^\ell$. 
We can also prove the following. 
\begin{lemma}
Let $F_1, F_2\in\scF(\A)$. Then 
$U_{F_1}\cap U_{F_2}\neq\emptyset$ if and only if 
$F_1$ and $F_2$ are comparable. 
\end{lemma}
Now we shall construct an open covering of $M(\A)$. 
Let $(X, C)\in\Sal(\A)$. Denote by $\widetilde{C}\in\ch(\A_X)$ 
the open convex cone corresponding to the unique chamber of 
$\A_X=\{H\in\A\mid H\supset X\}$ containing $C$. Note 
that $\widetilde{C}$ is an open convex cone. 
Associated with $(X, C)\in\Sal(\A)$, let us define 
\begin{equation}
Z_{(X, C)}:=U_X+\sqrt{-1}\cdot\widetilde{C}=\{x+\sqrt{-1}y\mid x\in U_X, y\in\widetilde{C}\}. 
\end{equation}
Then clearly, $Z_{(X, C)}\simeq U_X\times\widetilde{C}$ 
is an open convex subset in $M(\A)$. We can also prove 
the following. 
\begin{lemma}
\label{lem:Salchain}
\begin{itemize}
\item[(1)] 
If $(X_1, C_1)$ and $(X_2, C_2)$ are not comparable in $\Sal(\A)$, 
then $Z_{(X_1, C_1)}\cap Z_{(X_2, C_2)}=\emptyset$. 
\item[(2)] 
If $(X_1, C_1)\leq\cdots\leq (X_k, C_k)$ in $\Sal(\A)$. Then 
\[
Z_{(X_1, C_1)}\cap\cdots\cap Z_{(X_k, C_k)}=
\left(\bigcap_{i=1}^k U_{X_i}\right)\times \widetilde{C_k}. 
\]
In particular, 
$Z_{(X_1, C_1)}\cap\cdots\cap Z_{(X_k, C_k)}\neq\emptyset$ 
if and only if $(X_1, C_1), \dots, (X_k, C_k)$ form a chain in 
$\Sal(\A)$. 
\end{itemize}
\end{lemma}
By Lemma \ref{lem:Salchain} 
and the nerve lemma (\cite[Cor. 4G.3]{hatcher}, 
\cite[Theorem 15.21]{kozlov}), 
the complement $M(\A)$ is homotopy equivalent to 
$\Delta_{\Sal}(\A)$. Thus we obtain Theorem \ref{thm:salcpx}.

\subsection{Delucchi-Falk complex} 
\label{subsec:df}

Delucchi and Falk \cite{del-fal} introduced a complex which 
can be described by using only the metric structure of 
the adjacency 
graph $\Gamma(\A)$. The distance of two chambers 
$C_1, C_2\in\ch(\A)$ is 
\[
d(C_1, C_2)=\#\{H\in\A\mid H \mbox{ separates $C_1$ and $C_2$}\}. 
\]
\begin{definition}
Define the order on $\DF(\A):=\ch(\A)\times\ch(\A)$ by 
\[
(C, D)\leq (C', D')\Longleftrightarrow 
d(C', D')=d(C', C)+d(C, D)+d(D, D'). 
\]
\end{definition}
\begin{theorem}
(\cite{del-fal}) The order complex of $\DF(\A)$ is 
homotopy equivalent to $\Delta_{\Sal}(\A)$. 
\end{theorem}
\begin{proof}
The strategy is to construct a poset map $\Sal(\A)\longrightarrow
\DF(\A)$ that induces the homotopy equivalence of order 
complexes. Let $(X, C)\in\Sal(\A)$. Let $C'\in\ch(\A)$ be the 
opposite chamber of $C$ with respect to $X$. Then 
$(X, C)\longmapsto (C', C)$ gives a map 
$\Sal(\A)\longrightarrow\DF(\A)$, which induces a 
homotopy equivalence between order complexes. 
(See \cite{del-fal} for details.) 
\end{proof}

\begin{remark}
Notions in \S \ref{subsec:sal} and \S \ref{subsec:df} are generalized 
to oriented matroids \cite{ori-mat}. 
\end{remark}

\subsection{Minimality}

Let $X$ be a finite CW complex. In general, we have 
the following inequality 
\begin{equation}
\label{eq:lowerbd}
\#\{\mbox{$k$-dimensional cell}\}\geq b_k(X). 
\end{equation}
This means that the number of $k$-dimensional cells of 
a CW complex is bounded below by the $k$-th Betti number. 
A finite CW complex $X$ is called \emph{minimal} if the 
equality 
$\#\{\mbox{$k$-dimensional cell}\}\geq b_k(X)$ 
holds for every $k\geq 0$. This is equivalent to the condition that 
the boundary maps of the chain complex are vanishing. 
Hence, the $k$-dimensional cells of $X$ form 
a basis of the homology of $H_k(X, \Z)$. 
In general, a manifold does not have the homotopy type 
of a minimal CW complex. However, for the complement of 
a complex hyperplane arrangement, it is known to be minimal. 
\begin{proposition}
\label{prop:min}
(Dimca-Papadima \cite{dim-pap}, Randell \cite{ran-mor}) 
Let $\A$ be an essential hyperplane arrangement 
$\A$ in $\C^\ell$. 
Then the complement $M=M(\A)$ is homotopic to an 
$\ell$-dimensional minimcal CW complex. 
\end{proposition}
\begin{proof}[Proof of Proposition \ref{prop:min}]
The proof is by induction on $\ell$. If $\ell=1$, then 
$M(\A)$ is clearly homotopic to a bouquet of $n$ circles. 

Next, let $\ell>1$. Let $F\subset\C^\ell$ be a  generic hyperplane. 
By the affine Lefschetz hyperplane section theorem \cite{le-ham} 
(we will see later the idea of the proof of the affine Lefschetz 
hyperplane section theorem. See \S \ref{subsec:attach}), 
$M(\A)$ is homotopy equivalent to a space obtained by 
attaching $\ell$-dimensional cells to $M(\A)\cap F$. 
The number of $\ell$-cells is clearly 
$\rank H_\ell(M, M\cap F)$. 
Consider the homology long exact sequence of the pair of spaces 
$(M, M\cap F)$. We have 
\[
H_\ell(M\cap F)
\stackrel{(\alpha)}{\longrightarrow} 
H_\ell(M)
\stackrel{(\beta)}{\longrightarrow} 
H_\ell(M, M\cap F)
\stackrel{(\gamma)}{\longrightarrow} 
H_{\ell-1}(M\cap F)
\stackrel{(\delta)}{\longrightarrow} 
H_{\ell-1}(M)
\]
Since $M\cap F$ is homotopy equivalent to an 
$(\ell-1)$-dimensional CW complex, we have 
$H_\ell (M\cap F)=0$. By 
Proposition \ref{prop:inj}, $(\delta)$ is isomorphic. 
Therefore, both $(\alpha)$ and $(\gamma)$ are zero map, 
we have $(\beta)$ is isomorphic, and $b_\ell(M)=\rank H_\ell(M)$ 
is equal to the number of $\ell$-cells. 
\end{proof}

\subsection{Attaching maps}
\label{subsec:attach}

The homotopy type of the minimal CW complex depends on the 
attaching maps of cells. The affine Lefschetz hyperplane section 
theorem \cite{le-ham} tells us that there exist continuous maps 
\[
\sigma_i: (D^\ell, \partial D^\ell)\longrightarrow (M, M\cap F), 
\]
$(i=1, \dots, b_\ell(M))$, such that 
$M$ is homotopy equivalent to 
\[
(M\cap F)\cup\bigcup_{i=1}^{b_\ell(M)}D^\ell. 
\]
This fact can be proved using  Morse theory way as follows. 
Let $Q=\prod_{H\in\A}\alpha_H$ be the defining equation of 
$\A$. Let $F=\{f=0\}$, where $f:\C^\ell\longrightarrow\C$ is a 
defining equation of the generic hyperplane $F$. 
Consider 
\begin{equation}
\label{eq:morsefcn}
\varphi=\left|
\frac{f(x)^2}{Q(x)^{1/n}}
\right|
: M(\A)\longrightarrow\R_{\geq 0}
\end{equation}
By Varchenko conjecture (solved in \cite{sil, ot-var}), 
$\varphi$ has non-degenerate critical points with Morse index $\ell$, 
and the number of critical points is equal to 
\[
(-1)^\ell\chi(M\setminus F)=b_\ell(M(\A)). 
\]
(Note that, to apply solutions to Varchenko conjecture, we might 
have to perturb the exponents generically.) 
Recall that the gradient flow is a map 
$\phi: M\times \R\longrightarrow M$ 
satisfying 
\begin{equation}
\label{eq:gradflow}
\frac{\partial}{\partial t}\phi(x, t)=-\grad\varphi(x). 
\end{equation}
For each critical point $p\in\Crit(\varphi)$, define the 
stable cell $\scS_p$ and unstable cell $\scU_p$ as 
\[
\begin{split}
\scS_p&=\{x\in M\mid\lim_{t\to\infty}\phi(x, t)=p\}\\
\scU_p&=\{x\in M\mid\lim_{t\to-\infty}\phi(x, t)=p\}. 
\end{split}
\]
Then the boundary of $\scU_p$ is attached to $F$. These 
cells form the minimal CW complex. 
For the explicit description of the cells, there are two main 
problems: 
\begin{itemize}
\item[(a)] 
Describe the set of critical points. 
\item[(b)] 
Describe the unstable cells. 
\end{itemize}
The (a) is an algebraic problem. Therefore, it might be possible 
for individual arrangements. However, (b) is a more transcendental 
problem. It would be difficult to solve (b) even for a fixed 
arrangement. 

In the real arrangement case, let 
\[
\ch_F(\A)=\{C\in\ch(\A)\mid C\cap F=\emptyset\}. 
\]
Using the structure of chambers, 
one can solve the following weaker versions of (a) and (b) 
which are enough to describe the homotopy type. 
\begin{itemize}
\item[(a')] 
Describe the set of stable cells on which critical points exist. 
\item[(b')] 
Describe the homotopy type of unstable cells. 
\end{itemize}

\begin{proposition}
\cite[Theorem 4.3.1]{yos-lef} 
The Morse function $\varphi:M\to\R_{\geq 0}$ 
in (\ref{eq:morsefcn}) has a critical point $p_C$ in every 
chamber 
$C\in \ch_{\scF}^\ell(\A)$. 
Furthermore, the stable cell $\scS_{p_C}$ of $p_C$ 
is the chamber $C$ itself. 
\end{proposition}
\begin{proof}
Let $C\in\ch_F(\A)$. 
By the definition, $|\varphi(x)|\to\infty$ as 
$x\in C$ approaches the boundary. Hence 
$\varphi|_C:C\to\R$ has at least one minimal point $p_C\in C$. 
By the Cauchy-Riemann equation, $p_C$ is a critical point 
of $\varphi: M\to\R$. Now we have $b_\ell(M)$ critical 
points. However, by the Varchenko conjecture, there are no 
other critical points. 

The gradient flow of (\ref{eq:gradflow}) preserves real points, 
hence we have 
\[
C\subseteq\scS_{p_C}. 
\]
However, both are $\ell$-dimensional, thus we have 
$C=\scS_{p_C}$ (Figure \ref{fig:morsefcn}). 
\begin{figure}[htbp]
\centering
\begin{tikzpicture}


\draw[thick, ->] (-0.5,0.5) -- ++(0,3) node[right] {$\varphi$}; 

\draw[thick] (1,4) .. controls (1.5,4) and (3,4) .. (3,3.5);
\draw[thick] (0,3) .. controls (0.5,3) and (3,3) .. (3,3.5);

\draw[thick] (3,3.5) .. controls (3,2.5) and (4,2) .. (5,2);
\draw[thick] (7,3.5) .. controls (7,2.5) and (6,2) .. (5,2);

\draw[thick] (10,4) .. controls (9.5,4) and (7,4) .. (7,3.5);
\draw[thick] (9,3) .. controls (8.5,3) and (7,3) .. (7,3.5);

\draw[thick] (0,3) -- (0,0) -- (9,0) -- (9,3); 
\draw[thick] (1,4) -- ++(0, -0.9);
\draw[dashed] (1,3) -- (1,1) -- (9,1);
\draw[thick] (9,1) -- ++(1,0) -- ++(0,3);

\draw[thick] (6,0) .. controls (6,0.5) and (5.5,2) .. (5,2);
\draw[dashed] (4.3,1) .. controls (4.3,1.5) and (4.5,2) .. (5,2);

\coordinate (A) at (2.7,3.25); 
\coordinate (P) at (5.4,1.7); 
\coordinate (B) at (7.3,3.15); 

\draw[thick] (A) to [out=270, in=180] (P);
\draw[thick] (B) to [out=270, in=0] node[anchor=north west] {$\scS_{p_C}=C$} (P);
\draw (6,0) node[anchor=south west] {$\scU_{p_C}$}; 

\filldraw[fill=black, draw=black] (P) node [anchor=north east] {$p_C$} circle (0.07); 

\draw (10,1.5) -- ++(1,0) -- node[right] {$F=\{\varphi =0\}$} ++(-2,-2) --
++(-10,0) -- ++(1,1); 

\draw[ultra thin, dashed] (0,0.5) --++(1,1)--++(9,0); 

\end{tikzpicture}
\caption{$\scS_{p_C}$ and $\scU_{p_C}$} 
\label{fig:morsefcn}
\end{figure}
\end{proof}

\begin{proposition}
\cite[Corollary 5.1.5]{yos-lef} 
$X=M(\A)\setminus\bigsqcup_{C\in\ch_F(\A)}$ is diffeomorphic to 
$(M(\A)\cap F)\times D^2$. In particular, $X$ is homotopy 
equivalent to $M\cap F$. 
\end{proposition}
By these results, we conclude as follows. 
\begin{proposition}
\label{prop:characterizcells}
For each $C\in\ch_F(\A)$, the unstable cell 
$\sigma_C:(D^\ell, \partial D^\ell)\to (M, M\cap F)$ 
has the following properties. 
\begin{itemize}
\item[(i)] 
$\sigma_C$ is transversal to $C$ and 
$\sigma_C(D^\ell)\cap C$ is a point. 
\item[(ii)] 
$\sigma_C(D^\ell)\cap C'=\emptyset$ for 
$C'\in\ch_F(\A)\setminus\{C\}$. 
\end{itemize}
Furthermore, (i) and (ii) characterize the homotopy type 
of the map 
$\sigma_C:(D^\ell, \partial D^\ell)\to (M, M\cap F)$. 
\end{proposition}
One can construct an explicit continuous map 
$\sigma_C:(D^\ell, \partial D^\ell)\to (M, M\cap F)$ satisfying 
(i) and (ii) in Proposition \ref{prop:characterizcells}, see 
\cite{yos-lef} for details. For $2$-dimensional case, we 
will describe the dual stratification in \S \ref{sec:min2dim}. 
We also note that in $2$-dimensional case, one can describe 
the handle decomposition of the $4$-manifold $M(\A)$ by refining 
minimal CW decomposition \cite{sug-yos}.

\subsection{Minimality via Salvetti complex}

To avoid problems (a) and (b) in the previous section, 
an alternative way to describe the minimal CW complex is to 
start from the Salvetti complex. As we saw, the Salvetti 
complex is a large regular CW complex. For instance, the number 
of $0$-cells is equal to $\#\ch(\A)$. However, using 
discrete Morse theory \cite{kozlov, koz-col}, one can prove 
that the Salvetti complex $\Delta_{\Sal}(\A)$ is homotopy 
equivalent to a minimal CW complex. The idea is to construct 
an acyclic matching on the face poset of the complex. 
The remaining cells are called critical cells. 
Note that the collapse of cells along an acyclic matching 
does not change the homotopy type. 
It is sufficient to construct an acyclic matching matching such 
that the number of critical cells is equal to the sum of Betti 
numbers. This can be done in several ways. 
\begin{itemize}
\item 
Salvetti and Settepanella \cite{sal-sett} constructed an acyclic 
matching using a generic flag. 
\item 
Delucchi \cite{del-min}, Lofano and Paolini \cite{lof-pao} 
constructed a matching using the Euclidean distance 
between a generically fixed base point and chambers. 
This also proves the minimality for locally finite affine arrangements. 
\end{itemize}
These techniques are applicable to many other situations. 
For example, the minimality was proven for 
Salvetti comlexes for oriented matroids \cite{del-sett}, 
toric arrangements \cite{dan-del}, 
$2$-arrangements \cite{adi} and so on. 

\begin{remark}
The minimal CW complex described in \cite{yos-lef} and 
the one based on the Salvetti complex are very different. 
More precisely, the homotopy types of individual cells are 
different. It would be interesting to ask whether starting 
from Delucchi-Falk complex, can one describe both of them 
or not. 
\end{remark}

\subsection{Minimal stratification}
\label{sec:min2dim}

Let $\A=\{H_1, \dots, H_n\}$ be a line arrangement in $V=\R^2$. 
Fix a flag $F^0\subset F^1$ and defining equation $\alpha_i$ of 
$H_i$ as in \S \ref{sec:pres}. For the sake of 
simplicity, we also set $\alpha_0=-1$. 
Define the subset $S_i\subset M(\A)$ ($i=1, \dots, n$) as 
\begin{equation}
S_i=
\left\{
x\in M(\A)\left| \frac{\alpha_i(x)}{\alpha_{i-1}(x)}\in\R_{<0}
\right.
\right\}. 
\end{equation}
Note that since $\alpha_0=-1$, 
$S_1=\{x\in M(\A)\mid \alpha_1(x)\in\R_{>0}\}=
\alpha_1^{-1}(\R_{>0})$. Note also that if $x\in\R^2$, then 
$x\in S_i$ if and only if $\alpha_i(x)$ and $\alpha_{i-1}(x)$ have 
different signs, which means, roughly speaking, $x$ is in between 
$H_i$ and $H_{i-1}$. Each $S_i$ is a $3$-dimensional submanifold 
of $M(\A)$. Furthermore, $S_i$ and $S_j$ ($i\neq j$) 
intersect transversally, and 
$S_i\cap S_j$ is a union of chambers 
(Figure \ref{fig:3dimsub}). 
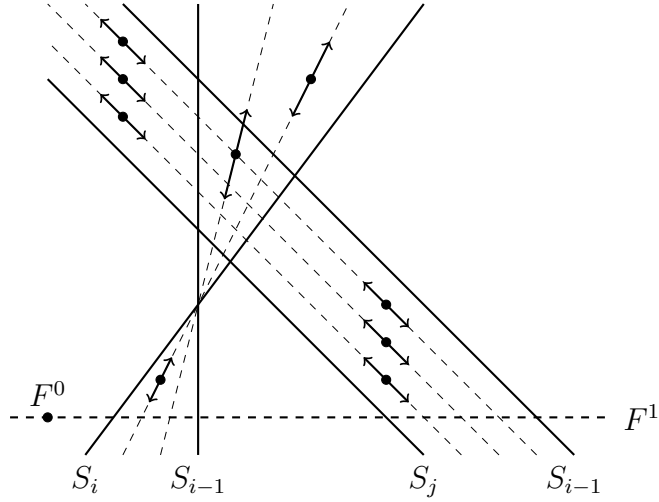
\begin{figure}[htbp]
\centering
\begin{tikzpicture}


\filldraw[fill=black, draw=black] (0,0.5) node [above] {$F^0$} circle (0.06); 
\draw[thick, dashed] (-0.5,0.5) -- ++(8,0) node [right] {$F^1$};

\draw[thick] (0.5,0) node[below] {$S_i$} -- ++(4.5,6);
\draw[thick] (2,0) node[below] {$S_{i-1}$} -- ++(0,6);
\draw[ultra thin, dashed] (1,0) -- ++(3,6);
\draw[ultra thin, dashed] (1.5,0) -- ++(1.5,6);

\coordinate (P1) at (1.5,1); 
\filldraw[fill=black, draw=black] (P1) circle (0.06); 
\draw[thick, ->] (P1) -- ++(0.15,0.3); 
\draw[thick, ->] (P1) -- ++(-0.15,-0.3); 

\coordinate (P2) at (3.5,5); 
\filldraw[fill=black, draw=black] (P2) circle (0.06); 
\draw[thick, ->] (P2) -- ++(0.25,0.5); 
\draw[thick, ->] (P2) -- ++(-0.25,-0.5); 

\coordinate (P3) at (2.5,4); 
\filldraw[fill=black, draw=black] (P3) circle (0.06); 
\draw[thick, ->] (P3) -- ++(0.15,0.6); 
\draw[thick, ->] (P3) -- ++(-0.15,-0.6); 

\draw[thick] (5,0) node[below] {$S_j$} -- ++(-5,5); 
\draw[thick] (7,0) node[below] {$S_{i-1}$} -- ++(-6,6); 
\draw[ultra thin, dashed] (5.5,0) -- ++(-5.5,5.5); 
\draw[ultra thin, dashed] (6,0) -- ++(-6,6); 
\draw[ultra thin, dashed] (6.5,0) -- ++(-6,6); 

\coordinate (P4) at (4.5,1); 
\filldraw[fill=black, draw=black] (P4) circle (0.06); 
\draw[thick, ->] (P4) -- ++(0.3,-0.3); 
\draw[thick, ->] (P4) -- ++(-0.3,0.3); 

\coordinate (P5) at (4.5,1.5); 
\filldraw[fill=black, draw=black] (P5) circle (0.06); 
\draw[thick, ->] (P5) -- ++(0.3,-0.3); 
\draw[thick, ->] (P5) -- ++(-0.3,0.3); 

\coordinate (P6) at (4.5,2); 
\filldraw[fill=black, draw=black] (P6) circle (0.06); 
\draw[thick, ->] (P6) -- ++(0.3,-0.3); 
\draw[thick, ->] (P6) -- ++(-0.3,0.3); 

\coordinate (P7) at (1,4.5); 
\filldraw[fill=black, draw=black] (P7) circle (0.06); 
\draw[thick, ->] (P7) -- ++(0.3,-0.3); 
\draw[thick, ->] (P7) -- ++(-0.3,0.3); 

\coordinate (P8) at (1,5); 
\filldraw[fill=black, draw=black] (P8) circle (0.06); 
\draw[thick, ->] (P8) -- ++(0.3,-0.3); 
\draw[thick, ->] (P8) -- ++(-0.3,0.3); 

\coordinate (P9) at (1,5.5); 
\filldraw[fill=black, draw=black] (P9) circle (0.06); 
\draw[thick, ->] (P9) -- ++(0.3,-0.3); 
\draw[thick, ->] (P9) -- ++(-0.3,0.3); 

\end{tikzpicture}
\caption{$S_i$ and $S_j$ ($i>j$)} 
\label{fig:3dimsub}
\end{figure}
The submanifolds $S_i$ ($i=1, \dots, n$) provide a stratification 
dual to minimal CW complex, in the sense that the number of 
codimension $k$ strata is equal to $b_k(M(\A))$ for $k=0, 1, 2$. 
\begin{proposition}
\label{prop:minstr}
\cite{yos-str}
Let $U=M(\A)\setminus\bigcup_{i=1}^n S_i$ and 
$S_i^\circ=S_i\setminus\bigcup_{C\in\ch_F(\A)}C$. Then, 
\begin{itemize}
\item[(1)] $S_1, \dots, S_n$ are $3$-dimensional 
contractible manifolds. 
\item[(2)] $U$ is a $4$-dimensional contractible manifold. 
\item[(3)] $M(\A)$ is decomposed as follows. 
\[
M(\A)=U\sqcup\bigsqcup_{i=1}^n S_i^\circ
\sqcup\bigsqcup_{C\in\ch_F(\A)}C. 
\]
\end{itemize}
\end{proposition}
From this stratification, one can obtain the 
presentation of the fundamental group in 
Proposition \ref{thm:minpres}.

\section{Local systems cohomology groups}

\subsection{Rank one local systems}

Let $G$ be an abelian group. 
Recall that a locally constant 
sheaf (or a local system), with local fiber $G$, on $M$ 
is a sheaf $\scL$ on $M$ such that for any point 
$p\in M$, there exists an open neighborhood $U$ of $p$ 
such that the restriction $\scL|_U$ is isomorphic to 
the constant sheaf associated to $G$. 
A local system is determined by a group homomorphism 
$\rho: \pi_1(M, p_0)\longrightarrow\Aut(G)$. 

In this section, we mainly consider the rank one local system 
over a commutative ring $R$. Namely, $\scL$ is a local system 
with local fiber isomorphic to $R$ as an $R$-module. Such a local 
system is determined by a homomorphism 
\[
\rho:\pi_1(M, p_0)\longrightarrow R^\times, 
\]
where $R^\times$ is the group of invertible elements. 
Since $R^\times$ is abelian, the map $\rho$ factors through 
the homology group, $\pi_1(M)\to H_1(M, \A)\to R^\times$, which 
is determined by assigning $\rho(\gamma_i)\in R^\times$ 
for each meridian $\gamma_i$ of the hyperplane $H_i$. 

For our purposes, it is convenient to realize a local system 
as a representation of the Deligne groupoid $\Gal(\A)$. 
Let $C_0$ and $C_1$ be  adjacent chambers separated by 
a hyperplane $H_i$. Then the sequence 
$\gamma_i=(C_0 C_1 C_0)\in\Hom(C_0, C_0)$ is corresponding to the 
meridian of $H_i$. We have $\rho(\gamma_i)\in R^\times$. 
To describe the representation of $\Gal(\A)$, we need 
$\rho((C_0, C_1))$ and $\rho((C_1, C_0))$. We set them as 
$x_i$ and $y_i$, respectively. More formally, define the extended 
ring $\widetilde{R}$ as follows. 
\begin{equation}
\widetilde{R}:=
R[x_a\mid a\in R^\times]/(x_a^2-a; a\in R^\times). 
\end{equation}
\begin{lemma}
(1) The natural map $i: R\longrightarrow \widetilde{R}$ is injective. 

(2) $x_a\in\widetilde{R}^\times$. 
\end{lemma}
\begin{proof}
(1) Let $r\in R$. Suppose $r\in\Ker i$. Then there exist 
$a_1, \dots, a_m\in R^\times$ and 
$f_1, \dots, f_m\in R[x_1, x_2, \dots, x_m]$ such that 
\[
r=\sum_{s=1}^m f_s(x_1, \dots, x_m)\cdot (x_s^2-a_s). 
\]
Let $A_k=R[x_1, \dots, x_k]/(x_s^2-a_s; s=1, \dots, k)$, 
for $k=1, \dots, m$. Then $r$ is contained in the kernel of 
$R\longrightarrow A_k$. 
Note that 
\[
A_{s}=A_{s-1}[x_s]/(x_s^2-a_s). 
\]
Since the ideal is generated by a monic polynomial, it is clear that 
the natural map $i_s: A_{s-1}\longrightarrow A_{s}$ is injective. 

(2) 
This is obvious from $x_a\cdot (x_a a^{-1})=1$ in $\widetilde{R}$. 
\end{proof}
From the above lemma, 
the multiplication map 
$R\longrightarrow Rx_a, r\longmapsto rx_a$ gives an isomorphism 
of $R$-modules. 

Let $\scF=(F^0\subset\cdots\subset F^{\ell-1})$ be a generic 
flag as in Remark \ref{rem:flagnearinfty}. 
Let us consider $F^0=p_0$ to be the base point of $M=M(\A)$. 
Let $\rho: \pi_1(M, p)\longrightarrow R^\times$ be a group 
homomorphism which is determined by 
$\rho(\gamma_i)=a_i\in R^\times$. Now we construct the 
representation of $\Gal(\A)$. 
Let $C\in\ch(\A)$. Then, define the submodule 
$\rho(C)\subset\widetilde{R}$ by 
\[
\rho(C):=R\cdot\prod_{H_i\in\Sep(C_0, C)}x_{a_i}, 
\]
where $C_0$ is the chamber containing $p_0=F^0$. 
\begin{proposition}
Denote the category of $R$-modules by $\Mod_R$. 
The correspondence $\rho: \ch(\A)\longrightarrow\Mod_R: 
C\longmapsto \rho(C)$ and $\rho: \Hom(C, C')\ni (C, C')\longmapsto 
x_{a_i}$, where $C$ and $C'$ are adjacent and separated by 
$H_i$, gives a representation of the groupoid $\Gal(\A)$. 
\end{proposition}
\begin{proof}
It is sufficient to show that the image $\rho(\gamma)$ of 
$\gamma\in\Hom(C, C)$ gives a automorphism 
of $\rho(C)$. Since $R^\times$ and $\widetilde{R}$ are abelian, 
each $x_{a_i}$ is multiplied an even number of times. 
In other words, the action of $\rho(\gamma)$ on $\rho(C)$ is 
a multiplication by $\prod_i x_{a_i}^{2k_i}$ ($k_i\in\Z$), 
which is equal to $\prod_i a_i^{k_i}\in R^\times$. 
\end{proof}

\begin{definition}
Let $C, C\in\ch(\A)$. Define $\Delta(C, C')\in\widetilde{R}$ 
by 
\[
\Delta(C, C')=
\prod_{H_i\in\Sep(C, C')} x_{a_i}-
\prod_{H_i\in\Sep(C, C')} x_{a_i}^{-1}. 
\]
\end{definition}

\subsection{The degree map and twisted cochain complex}

To describe the twisted cochain complex, we need the notion of 
the degree map \cite{yos-lef, bai-yos, yos-ch}
\[
\deg: \ch_{\scF}^k(\A)\times \ch_{\scF}^{k+1}(\A)\longrightarrow\Z. 
\]
Let $B=B^k\subset F^k$ be a $k$-dimensional ball of 
sufficiently large radius such that every $0$-dimensional 
intersection $X\in L_0(\A\cap F^k)$ is contained in the interior 
of $B$. Let $C'\in\ch_{\scF}^{k+1}(\A)$. Then, consider the 
vector field $U^{C'}$ on $B$, that is $U^{C'}(x)\in T_xF^k$ for 
$x\in B$, satisfying the following conditions. 
\begin{itemize}
\item 
$U^{C'}(x)\neq 0$ for $x\in \partial(\overline{C}\cap B)$ with 
$C\in \ch_{\scF}^k(\A)$. 
\item 
Let $x\in\partial B$, then $U^{C'}(x)$ directs to the inside of $B$. 
\item 
If $x\in H$ for some $H\in\A$, then $U^{C'}(x)\notin T_xH$ and 
directs the side containing $C'$. 
\end{itemize}
Roughly speaking $U^{C'}$ is a vector field on $B$ directing 
towards the chamber $C'$. 
\begin{definition}
Define $\deg(C, C')\in\Z$ by the degree of 
the Gauss map on the boundary $\partial(\overline{C}\cap B)$
\[
\deg(C, C'):=
\deg
\left(
\left.
\frac{U^{C'}}{|U^{C'}|}
\right|_{\partial(\overline{C}\cap B)}: 
\partial(\overline{C}\cap B)\longrightarrow S^{k-1}
\right). 
\]
\end{definition}
The next result was first formulated for 
$R=\C$ in \cite{yos-lef}. Recently, the case $R=\Z$ was 
considered in \cite{sug-cdo}. 
\begin{proposition}
\label{prop:twimincpx}
\cite{yos-lef}
Let $\rho:\pi_1(M(\A))\longrightarrow R^\times$ be an $R$-rank one 
local system. Let 
\[
\scC^k(\A, \rho):=\bigoplus_{C\in\ch_\scF^k(\A)}\rho(C). 
\]
Define the map $\nabla_\rho: 
\scC^k(\A, \rho)\longrightarrow \scC^{k+1}(\A, \rho)$ by 
\[
\nabla_\rho(y)=
\sum_{C'\in\ch_\scF^{k+1}}\deg(C, C')\cdot y \cdot\Delta(C, C'), 
\]
where $y\in\rho(C)$. Then 
$(\scC^\bullet (\A, \rho), \nabla_\rho)$ gives a cochain complex 
such that 
\[
H^k(\scC^\bullet (\A, \rho), \nabla_\rho)\simeq 
H^k(M(\A), \scL_\rho). 
\]
\end{proposition}
\begin{remark}
Proposition \ref{prop:twimincpx} has been applied 
to the computation of the first Betti number of 
the Milnor fiber \cite{yos-mil, yos-dou}, 
complex rank one local systems \cite{yos-vie, yos-q, yos-loc}, 
and Aomoto complexes \cite{bai-yos, ty-res}. 
Recently, Sugawara \cite{sug-cdo} used it to settle 
an integral version of Cohen-Dimca-Orlik type 
vanishing theorem \cite{cdo}. See also 
for an alternative approach \cite{lmw}. 
\end{remark}

\section{Homotopy groups}

\subsection{$K(\pi, 1)$ arrangements}

Recall that a connected topological space $X$ is called 
$K(\pi, 1)$ if the higher homotopy groups are vanishing, 
namely, $\pi_k(X)=0$ for $k\geq 2$. 
Naturally, we call a complex arrangement $\A$ a $K(\pi, 1)$ 
arrangement if the complement $M(\A)$ is a $K(\pi, 1)$ space. 
As we saw in Example \ref{ex:cbraid}, the braid arrangement 
is $K(\pi, 1)$. Based on this and other examples, 
Brieskorn \cite{bri} conjectured the $K(\pi, 1)$-ness of 
certain arrangements related to reflection groups. 
The $K(\pi, 1)$-ness has been settled in many cases, 
for example, 
\begin{itemize}
\item 
Simplicial arrangements \cite{del-kpi1}, 
\item 
Many $2$-dimensional real arrangements \cite{fal-kpi1}, 
\item 
Many reflection groups \cite{cha-dav}, 
\item 
Complex reflection groups \cite{bessis}, 
\item 
Affine Weyl arrangements \cite{pao-sal}, 
\end{itemize}
and so on. 
See \cite{fr-hom1, fr-hom2, par-conj} for surveys on 
this topic. 

\subsection{Localization and Hurewicz map}

We recall several basic facts about homotopy groups of $M(\A)$. 
The following is essentially due to M. Oka  
\cite[Lemma 1.1.]{par-del}. 
\begin{proposition}
Let $A$ be an arrangement in $\C^\ell$. Let $X\in L(\A)$ and 
$i: M(\A)\hookrightarrow M(\A_X)$ be the natural inclusion. 
Then $i$ induces the split surjective homomorphism 
\[
i_*: \pi_k(M(\A))\longrightarrow \pi_k(M(\A_X))
\]
of groups for $k\geq 1$. In other words, there exists a group 
injection $f:\pi_k(M(\A_X))\longrightarrow \pi_k(M(\A))$ such that 
$i_*\circ f:\pi_k(M(\A_X))\longrightarrow \pi_k(M(\A_X))$ is 
the identity. 
\end{proposition}
\begin{proof}
Let $p\in X\setminus\bigcup_{H\in\A\smallsetminus\A_X}H$. 
Take a small convex open neighborhood $U$ of $p$ in such a 
way that $B\cap H=\emptyset$ for $H\in\A\smallsetminus\A_X$. 
Then 
we have inclusions $j:U\cap M(\A)\hookrightarrow M(\A)$ 
and $i: M(\A) \hookrightarrow M(\A_X)$. Then 
$i\circ j: U\cap M(\A)\longrightarrow M(\A)$ is 
a homotopy equivalence. Hence the composition of induced 
homomorphism 
\[
\pi_k(M(\A_X))
\stackrel{j_*}{\longrightarrow}
\pi_k(M(\A))
\stackrel{i_*}{\longrightarrow}
\pi_k(M(\A_X))
\]
is isomorphic. 
\end{proof}
\begin{corollary}
(1) The fundamental group $\pi_1(M(\A_X))$ embeds into 
$\pi_1(M(\A))$. 

(2) If $\A$ is $K(\pi, 1)$, so is the localization $\A_X$. 
\end{corollary}
One of the difficulties of studying the homotopy groups of $M(\A)$ 
is the following vanishing of Hurewicz maps. One can not 
detect elements of homotopy groups as elements in the 
homology groups. 

\begin{proposition}
\label{prop:randell}
(Randell \cite{ran-hom}) 
Let $\A$ be an arrangement in $\C^\ell$. Then for any 
$k\geq 2$, the Hurewicz map 
\[
h: \pi_k(M)\longrightarrow H_k(M, \Z)
\]
is vanishing. 
\end{proposition}
\begin{proof}
Recall (\ref{eq:tuple}) that $M(\A)$ can be embedded into 
$(\C^\times)^n$. Note that 
$(\C^\times)^n$ is a $K(\Z^n, 1)$-space, so we have 
$\pi_k((\C^\times)^n)=0$, for $k\geq 2$. By the functoriality and 
naturality of Hurewicz maps, the composition 
$\bm{\alpha}_*\circ h: \pi_k(M(\A))\longrightarrow 
H_k((\C^\times)^n)$ is vanishing. By the 
injectivity 
(Proposition \ref{prop:inj}) of $\bm{\alpha}_*$, we conclude that 
the Hurewicz map $h: \pi_k(M)\longrightarrow H_k(M, \Z)$ 
is vanishing. 
\end{proof}

\subsection{Non-$K(\pi, 1)$ arrangements} 

There are many non-$K(\pi, 1)$ arrangements. 
Hattori \cite{hat75} proved that if $\A=\{H_1, \dots, H_n\}$ 
is a generic arrangement in $\C^\ell$ ($n\geq \ell$), then 
$M(\A)$ is homotopy equivalent to the $\ell$-skeleton 
of the standard CW structure of $n$-torus $T^n=(S^1)^n$, 
which clearly has $\pi_\ell(M(\A))\neq 0$. Papadima and Suciu 
\cite{ps-h} studied the structure of higher homotopy groups 
in a generalized setting. One of the most general results 
concerning non-vanishing of homotopy groups 
is the following. 
The proof in this article was due to M. Falk and S. Papadima. 
\begin{proposition}
\label{prop:gensec}
\cite{yos-gen} 
Let $\A$ be an essential arrangement in 
$V=\C^\ell$ with $\ell\geq 3$. 
Let $F\subset V$ be a $k$-dimensional generic subspace 
with $k\geq 2$. 
Then the arrangement $\A\cap F$ is not $K(\pi, 1)$. 
More precisely, $\pi_k(M(\A)\cap F)\neq 0$. 
\end{proposition}
\begin{proof}
It is sufficient to prove the case $k=\ell-1$. Let $F$ be a 
generic hyperplane. Then we will prove that 
$\pi_{\ell-1}(M(\A)\cap F)\neq 0$. 

Recall (Proposition \ref{prop:min}) that $M(\A)$ is 
homotopy equivalent to a space obtained by attaching 
$\ell$-cells to $M(\A)\cap F$. 
Suppose that $\pi_{\ell-1}(M(\A)\cap F)= 0$. Then the 
attaching maps $\partial D^\ell=S^{\ell-1}$ of 
$\ell$-cells must be trivial, and we have 
\[
M(\A)\simeq (M(\A)\cap F)\vee\bigvee_{i=1}^{b_\ell(M)}S^\ell. 
\]
Thus we have a map 
$i: \bigvee_{i=1}^{b}S^\ell\longrightarrow M(\A)$ 
(where $b=b_\ell(M(\A))$) that induces the surjection 
$i^*: H^\ell(M(\A))\longrightarrow 
H^\ell(\bigvee_{i=1}^{b}S^\ell)=\Z^b$. 
Consider the following diagram. 
\[
\begin{CD}
\wedge^\ell H^1(M(\A)) @>{(p)}>> \wedge^\ell H^1(\bigvee_{i=1}^{b}S^\ell)=0 \\
  @V{(q)}VV    @VV{(r)}V \\
H^\ell(M(\A)) @>{(s)}>> H^\ell(\bigvee_{i=1}^{b}S^\ell) =\Z^b
\end{CD}
\]
Then, by Orlik-Solomon's presentation of the cohomology ring 
(Proposition \ref{prop:OSPoin}), the map $(q)$ is surjective. 
Note that $(s)=i^*$ is also surjective. 
Hence, the composition $(s)\circ (q)$ is also surjective, 
which contradicts $(r)\circ (p)=0$. 
\end{proof}

\begin{proposition}
\label{prop:abgr}
\cite{hat75, fr-hom1}
Let $\A$ be an essential arrangement in $\C^\ell$. 
If there exists a subgroup $\Gamma\subset\pi_1(M(\A))$ 
such that $\Gamma\simeq\Z^{\ell+1}$, 
then $\A$ is not $K(\pi, 1)$. 
\end{proposition}
\begin{proof}
Let $K$ be an $\ell$-dimensional CW complex which is 
homotopy equivalent to $M=M(\A)$. Suppose $K$ is 
$K(\pi, 1)$. Then the universal cover $\widetilde{K}$ is 
contractible space, on which $\Gamma$ acts freely. Thus 
$K':=\widetilde{K}/\Gamma$ is a $K(\Gamma, 1)$ space. 
Since $K'$ is $\ell$-dimensional, $H_{\ell+1}(K')=0$. However, 
this contradicts the fact that 
$H_{\ell+1}(\Gamma, \Z)=H_{\ell+1}((S^1)^{\ell}, \Z)=\Z$. 
\end{proof}

\begin{example}
Let $\A=\{H_1, \dots, H_6\}$ be the arrangement in 
Figure \ref{fig:flag}. See Example \ref{ex:X3pres} 
for presentations of $\pi_1(M(\A))$. One can directly check 
that $\gamma_2, \gamma_5, [\gamma_3, \gamma_4]$ are 
commutative with each other and generate a subgroup 
isomorphic to $\Z^3$. Hence $M(\A)$ is not $K(\pi, 1)$. 
\end{example}

\subsection{Twisted Hurewicz maps and homotopy groups}

In this section, we recall the notion of the twisted Hurewicz map. 
Let $X$ be a CW complex, $\scL$ be a local system on $X$. 
Let $f: (S^k, *) \longrightarrow (X, x_0)$ be a continuous 
map from the sphere $S^k$ with $k\geq 2$. Then the pull-back 
$f^*\scL$ is a local system on $S^k$. Since $S^k$ is 
simply connected, $f^*\scL$ is a trivial local system. 
Hence the set of global sections $\Gamma(S^k, f^*\scL)$ is 
isomorphic to that of the local sections $\scL_{x_0}$ at $x_0$. 

Since $S^k$ (equipped with an orientation) is closed, 
the continuous map $f: (S^k, *) \longrightarrow (X, x_0)$ 
together with a 
global section $\sigma\in \Gamma(S^k, f^*\scL)$ 
gives a twisted cycle $[f]\otimes\sigma$ on $X$. 
Thus we have a map 
\[
h_{\scL, x_0}: 
\pi_k(X, x_0)\otimes\scL_{x_0}\longrightarrow 
H_k(X, \scL), 
\]
which we call the twisted Hurewicz map. Obviously, when $\scL$ 
is the trivial local system $\Z$, it is the classical Hurewicz map. 

As we have already seen in Proposition \ref{prop:randell}, 
the classical Hurewicz map does not carry any information on 
the higher homotopy groups. However, the twisted Hurewicz map 
can detect some nontrivial elements in the homotopy groups. 
Actually, in the setting of Proposition \ref{prop:gensec}, 
the twisted Hurewicz map become surjective. 

\begin{proposition}
\label{prop:twihure}
\cite{yos-gen} 
Let $\A$ be an essential arrangement in 
$V=\C^\ell$ with $\ell\geq 3$. 
Let $F\subset V$ be a $k$-dimensional generic subspace 
with $k\geq 2$. Let $\scL$ be a generic 
complex rank one local system. Then the twisted Hurewicz map 
\[
h_{\scL, x_0}: 
\pi_k(M(\A)\cap F, x_0)\otimes\scL_{x_0}\longrightarrow 
H_k(M(\A)\cap F, \scL) 
\]
is surjective (and $H_k(M(\A), \scL)$ is nonzero). 
\end{proposition}
The proof of Proposition \ref{prop:twihure} is based on 
the vanishing theorem of local system homology groups 
\cite{koh, esv, stv, cdo}. For certain generic local systems, 
$H_k(M(\A), \scL)$ vanishes except for $k=\ell$ and the 
dimension of the remaining cohomology is equal to the 
Euler characteristic of $M(\A)$. 
The point of the proof is that 
the non-vanishing degrees of $H_k(M(\A), \scL)$ and 
$H_k(M(\A)\cap F, \scL|_{M(\A)\cap F})$ differ by one 
if $F$ is a generic hyperplane. 
To cancel the twisted cycle of $H_{\ell-1}(M\cap F, \scL)$, 
the twisted Hurewicz map needs to be surjective. 


We now discuss an explicit construction of a sphere that 
is non-trivial in the homotopy group. The setting is as follows. 
Let $\A=\{H_1, \dots, H_n\}$ be a central and essential 
arrangement in $V=\R^\ell$ with $\ell\geq 3$. 
Let $S=S^{\ell-1}$ be the unit sphere centered at the origin. 
Consider the shift of the sphere $S$ into imaginary direction. 
Let $\theta: S\ni x\longmapsto \theta(x)\in T_xV$ 
be a vector field. Using the description of $M(\A)$ in 
\S \ref{sec:cpxf}, we consider the sphere 
\begin{equation}
S(\theta)=\{x+\sqrt{-1}\cdot\theta(x)\mid x\in S\}. 
\end{equation}
Then $S(\theta)$ is embedded into $M(\A)$ if and only if 
\begin{equation}
\label{eq:notincond}
\theta(x)\notin T_xH_i 
\end{equation}
for any $x\in H_i$, $H_i\in\A$. 
Since $\ell\geq 3$, $H_i\cap S$ is connected. Therefore, 
$\theta$ determines a half-space 
$H_i^{\varepsilon_i(\theta)}$, and a sign vector 
$(\varepsilon_1(\theta), \varepsilon_2(\theta), \dots, 
\varepsilon_{n}(\theta))\in\{\pm\}^n$ (Figure \ref{fig:shifted}).

\begin{figure}[htbp]
\centering
\begin{tikzpicture}


\draw[thick] (0,0) circle (3);


\draw[thin, dotted] (0:3 and 0.5) arc (0:180: 3 and 0.5);
\draw[thick] (180:3 and 0.5) arc (180:360: 3 and 0.5);
\draw[thin, ->] (220:3 and 0.5) -- ++(-0.5,-0.3);
\draw[thin, ->] (230:3 and 0.5) -- ++(-0.5,-0.3);
\draw[thin, ->] (240:3 and 0.5) -- ++(-0.5,-0.3);
\draw[thin, ->] (250:3 and 0.5) -- ++(-0.4,-0.3);
\draw[thin, ->] (256:3 and 0.5) -- ++(-0.4,-0.3);
\draw[thin, ->] (260:3 and 0.5) -- ++(-0.2,-0.4);
\draw[thin, ->] (265:3 and 0.5) -- ++(-0,-0.4);
\draw[thin, ->] (270:3 and 0.5) -- ++(0.2,-0.4);
\draw[thin, ->] (275:3 and 0.5) -- ++(0.4,-0.3);
\draw[thin, ->] (283:3 and 0.5) -- ++(0.4,-0.3);
\draw[thin, ->] (290:3 and 0.5) -- ++(0.4,-0.3);
\draw[thin, ->] (300:3 and 0.5) -- ++(0.2,-0.3);
\draw[thin, ->] (310:3 and 0.5) -- ++(0,-0.3);
\draw[thin, ->] (320:3 and 0.5) -- ++(-0.2,-0.3);
\draw[thin, ->] (330:3 and 0.5) -- ++(-0.3,-0.3);
\draw[thin, ->] (340:3 and 0.5) -- ++(-0.3,-0.3);

\draw[thin, dotted, rotate=240] (0:3 and 0.4) arc (0:180: 3 and 0.4);
\draw[thick, rotate=240] (180:3 and 0.4) arc (180:360: 3 and 0.4);
\draw[thin, ->, rotate=240] (200:3 and 0.4) -- ++(0.2, -0.4);
\draw[thin, ->, rotate=240] (210:3 and 0.4) -- ++(0.2, -0.4);
\draw[thin, ->, rotate=240] (220:3 and 0.4) -- ++(0.2, -0.4);
\draw[thin, ->, rotate=240] (230:3 and 0.4) -- ++(0.3, -0.3);
\draw[thin, ->, rotate=240] (240:3 and 0.4) -- ++(0.1, -0.3);
\draw[thin, ->, rotate=240] (250:3 and 0.4) -- ++(-0.2, -0.4);
\draw[thin, ->, rotate=240] (258:3 and 0.4) -- ++(-0.4, -0.3);
\draw[thin, ->, rotate=240] (270:3 and 0.4) -- ++(0, -0.4);
\draw[thin, ->, rotate=240] (280:3 and 0.4) -- ++(0.2, -0.3);
\draw[thin, ->, rotate=240] (300:3 and 0.4) -- ++(0.4, -0.2);
\draw[thin, ->, rotate=240] (310:3 and 0.4) -- ++(0.4, -0.2);
\draw[thin, ->, rotate=240] (320:3 and 0.4) -- ++(0.4, -0.2);

\draw[thin, dotted, rotate=120] (0:3 and 0.3) arc (0:180: 3 and 0.3);
\draw[thick, rotate=120] (180:3 and 0.3) arc (180:360: 3 and 0.3);
\draw[thin, ->, rotate=120] (220:3 and 0.3) -- ++(-0.5,-0.3);
\draw[thin, ->, rotate=120] (230:3 and 0.3) -- ++(-0.5,-0.3);
\draw[thin, ->, rotate=120] (240:3 and 0.3) -- ++(-0.5,-0.3);
\draw[thin, ->, rotate=120] (247:3 and 0.3) -- ++(-0.4,-0.2);
\draw[thin, ->, rotate=120] (260:3 and 0.3) -- ++(-0.2,-0.4);
\draw[thin, ->, rotate=120] (265:3 and 0.3) -- ++(-0,-0.4);
\draw[thin, ->, rotate=120] (270:3 and 0.3) -- ++(0.2,-0.4);
\draw[thin, ->, rotate=120] (275:3 and 0.3) -- ++(0.4,-0.3);
\draw[thin, ->, rotate=120] (290:3 and 0.3) -- ++(0.2,-0.3);
\draw[thin, ->, rotate=120] (300:3 and 0.3) -- ++(0,-0.3);
\draw[thin, ->, rotate=120] (310:3 and 0.3) -- ++(-0.3,-0.3);
\draw[thin, ->, rotate=120] (320:3 and 0.3) -- ++(-0.4,-0.3);
\draw[thin, ->, rotate=120] (330:3 and 0.3) -- ++(-0.4,-0.3);
\draw[thin, ->, rotate=120] (340:3 and 0.3) -- ++(-0.4,-0.3);

\draw[thin, dotted, rotate=165] (0:3 and 2) arc (0:180: 3 and 2);
\draw[thick, rotate=165] (180:3 and 2) arc (180:360: 3 and 2);
\draw[thin, ->, rotate=165] (210:3 and 2) -- ++(0.4,0.3);
\draw[thin, ->, rotate=165] (220:3 and 2) -- ++(0.4,0.3);
\draw[thin, ->, rotate=165] (230:3 and 2) -- ++(0.4,0.3);
\draw[thin, ->, rotate=165] (240:3 and 2) -- ++(0.4,0.2);
\draw[thin, ->, rotate=165] (250:3 and 2) -- ++(0.4,0.2);
\draw[thin, ->, rotate=165] (256:3 and 2) -- ++(0.4,0.2);
\draw[thin, ->, rotate=165] (277:3 and 2) -- ++(0.1,0.3);
\draw[thin, ->, rotate=165] (290:3 and 2) -- ++(-0.3,0.2);
\draw[thin, ->, rotate=165] (310:3 and 2) -- ++(-0.3,0.2);
\draw[thin, ->, rotate=165] (320:3 and 2) -- ++(-0.2,0.3);
\draw[thin, ->, rotate=165] (330:3 and 2) -- ++(-0.3,0.3);
\draw[thin, ->, rotate=165] (340:3 and 2) -- ++(-0.3,0.3);

\end{tikzpicture}
\caption{Shifted sphere} 
\label{fig:shifted}
\end{figure}
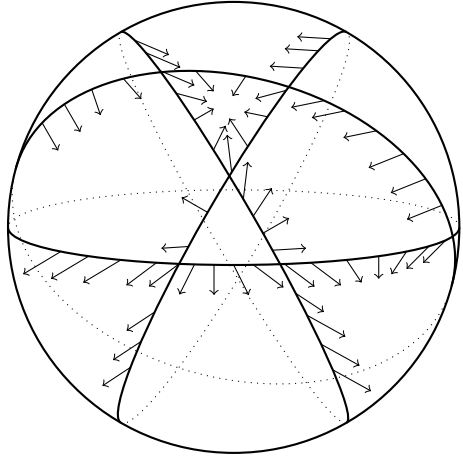
The (non)triviality of $[S(\theta)]\in\pi_{\ell-1}(M(\A))$ is 
checked as follows. 

\begin{proposition}
\label{prop:embsph}
\cite{yos-const} 
$[S(\theta)]=0$ in $\pi_{\ell-1}(M(\A))$ is if and only if 
the intersection of half-spaces 
$\bigcap_{i=1}^nH_i^{\varepsilon_i(\theta)}$ is non-empty. 
\end{proposition}
The ``if'' part is relatively easy. Indeed, if the intersection of 
half-spaces 
$\bigcap_{i=1}^nH_i^{\varepsilon_i(\theta)}=C$ is a chamber, 
one can extend the vector field $\theta$ to that of the ball 
$\widetilde{\theta}: D^\ell\longrightarrow TV$ in such a way 
that the condition (\ref{eq:notincond}) is fulfilled. 

The proof of ``only if'' part needs the twisted Hurewicz map. 
Suppose $\bigcap_{i=1}^nH_i^{\varepsilon_i(\theta)}=\emptyset$ 
(see Figure \ref{fig:shifted} for an example). 
Let $\scL$ be the complex rank one local system defined by 
$\rho: \pi_1(M(\A))\longrightarrow\C^\times$, 
$\gamma_i\longmapsto e^{2\pi\sqrt{-1}/n}$, where 
$\gamma_i$ is the meridian of $H_i$. Then the image of 
$[S(\theta)]$ by the twisted Hurewicz map 
\[
h_{\scL, x_0}: \pi_{\ell-1}(M(\A), x_0)\otimes\scL_{x_0}
\longrightarrow
H_{\ell-1}(M(\A), \scL)
\]
is non-zero. More precisely, one can construct an 
explicit twisted Borel-Moore cycle $[W]\in H_{\ell+1}^{BM}
(M(\A), \scL^{\vee})$ such that the twisted intersection number 
is nonzero. See \cite{yos-const} for details.

\medskip

\emph{Acknowledgements.} 
 This work was partially supported by JSPS KAKENHI 
(Grant No. JP23H00081).

\end{document}